\documentclass{article}
\usepackage{graphicx}
\usepackage{amsmath,amsthm,amssymb,pifont,colortbl,amscd, wrapfig}
\usepackage{bigdelim,multirow}
\newtheorem{theorem}{Theorem}[section]
\newtheorem{lemma}[theorem]{Lemma}
\newtheorem{proposition}[theorem]{Proposition}

\newtheorem{corollary}[theorem]{Corollary}
\newtheorem{conjecture}[theorem]{Conjecture}

\theoremstyle{definition}
\newtheorem{definition}[theorem]{Definition}
\newtheorem{remark}[theorem]{Remark}
\newtheorem*{acknowledgments}{Acknowledgments}

\def\deg{\mathop{\mathrm{deg}}\nolimits}

\def\ad{\mathop{\mathrm{ad}}\nolimits}
\def\tr{\mathop{\mathrm{tr}}\nolimits}
\def\End{\mathop{\mathrm{End}}\nolimits}
\def\id{\mathop{\mathrm{id}}\nolimits}

\def\Span{\mathop{\mathrm{Span}}\nolimits}
\def\Span{\mathop{\mathrm{Span}}\nolimits}
\begin{document}
\title{On the universal $sl_2$ invariant of ribbon bottom tangles}
\author{Sakie Suzuki\thanks{Research Institute for Mathematical Sciences, Kyoto
University, Kyoto, 606-8502, Japan. E-mail address: \texttt{sakie@kurims.kyoto-u.ac.jp}} }
\date{May 12, 2009}

\maketitle

\begin{center}
\textbf{Abstract}
\end{center}
A \textit{bottom tangle} is a tangle in a cube consisting of  arc components whose boundary points are on a line in the bottom square of the cube.
A \textit{ribbon bottom tangle} is a bottom tangle whose \textit{closure} is a ribbon link. For every $n$-component ribbon bottom tangle $T$, we prove that
the universal invariant $J_T$ of $T$ associated to the quantized enveloping algebra $U_h(sl_2)$ of the Lie algebra $sl_2$ is contained in a certain  $\mathbb{Z}[q,q^{-1}]$-subalgebra of the $n$-fold completed tensor power $U_h^{\hat  {\otimes }n}(sl_2)$ of $U_h(sl_2)$.
This result is  applied  to the colored Jones polynomial of ribbon links.

\section{Introduction.}

For  each ribbon Hopf algebra $H$, Reshetikhin and Turaev \cite{Re} defined  invariants
of  framed links colored by finite dimensional representations.
A  \textit{universal invariant} \cite{R1,R2,O} associated to $H$  is an invariant of framed tangles and links defined without using representations.
The universal invariant has a universality property such that
the colored link invariants constructed by  Reshetikhin and Turaev are obtained from the universal invariants by taking trace in the representations attached to the components of links.

A quantized enveloping algebra $U_h:=U_h(sl_2)$ of the  Lie algebra $sl_2$ is  a complete ribbon Hopf $\mathbb{Q}[[h]]$-algebra.
By the \textit{universal $sl_2$ invariant}, we mean the universal invariant associated to $U_h$. 
In \cite{H1}, Habiro  studied the  universal invariant of  \textit{bottom tangles} (see Section \ref{bot})  associated to an arbitrary ribbon Hopf algebra, and  in \cite{H2}, he  studied the universal $sl_2$ invariant of bottom tangles (see Section \ref{bottom}).
The universal $sl_2$ invariant  of an $n$-component bottom tangle  takes values in  the $n$-fold completed tensor power $U_h^{\hat  {\otimes }n}$ of $U_h$.
For every oriented, ordered, framed link $L$,  there is  a bottom tangle whose \textit{closure} is isotopic to $L$.
The universal invariant of bottom tangles has a universality property such that 
the colored link invariants  of a link $L$ is obtained from the  universal  invariant of a bottom tangle $T$ whose closure is isotopic to  $L$,  by taking the  quantum trace in the representations attached to the components of links.
In particular, one can  obtain the colored Jones polynomials of links from the  universal  $sl_2$ invariant of  bottom tangles.

An $n$-component link $L$ is called  a \textit{ribbon link}  if it bounds a system of $n$ ribbon disks in $S^3$.
Mizuma  \cite{Mi}  derived an explicit formula for the first derivative at $-1$ for the Jones polynomial of  $1$-fusion ribbon knots,
and    in \cite{Mi2},  she estimated the ribbon number of those knots by using the formula.
Eisermann \cite{Ei} proved that the Jones polynomial $V(L)\in \mathbb{Z}[v,v^{-1}]$ of an $n$-component
ribbon link $L$ is divisible by the Jones polynomial $V(O^n)=(v+v^{-1})^n$ of the  $n$-component unlink $O^n$.

A \textit{ribbon bottom tangle} is defined as a bottom tangle whose closure is a ribbon link.
In this paper, we study the universal $sl_2$ invariant  of   ribbon bottom tangles.

\subsection{Main result.}
Set $v=\exp\frac{h}{2}$, $q=v^2.$ We have $\mathbb{Z}[q,q^{-1}]\subset \mathbb{Z}[v,v^{-1}]\subset \mathbb{Q}[[h]]$.
Let $J_T$ denote the universal $sl_2$ invariant of a bottom tangle $T$.

Habiro \cite{H2}  proved that the universal $sl_2$ invariant $J_T$ of  an $n$-component,
 algebraically-split, $0$-framed bottom tangle $T$  is contained in a certain $\mathbb{Z}[q,q^{-1}]$-subalgebra 
$(\tilde {\mathcal{U}}_q^{ev})^{\tilde \otimes n}$ of $U_h^{\hat \otimes n}$.
He also defined another  $\mathbb{Z}[q,q^{-1}]$-subalgebra  $ (\bar U_q^{ev})\;\tilde {}^{\;\tilde \otimes n}\subset (\tilde {\mathcal{U}}_q^{ev})^{\tilde \otimes n}$ 
 and stated the following conjecture for boundary bottom tangle. (A bottom tangle is said to be \textit{boundary} if it bounds mutually disjoint Seifert surfaces in $[0,1]^3$, see  \cite{H1} for the detail.)
\begin{conjecture}[Habiro \cite{H2}]\label{Habico}
Let $T$ be an $n$-component boundary bottom tangle with $0$-framing.
Then we have $J_T\in  (\bar U_q^{ev})\;\tilde {}^{\;\tilde \otimes n}$.
\end{conjecture}
We shall define another subalgebra  $(\bar U_q^{ev})\;\hat  {}^{\;\hat  \otimes n}\subset (\bar U_q^{ev})\;\tilde {}^{\;\tilde \otimes n}$. 
(Here, we do not know whether the inclusion is proper or not, but the definition of $(\bar U_q^{ev})\;\hat  {}^{\;\hat  \otimes n}$ is more natural than that of $(\bar U_q^{ev})\;\tilde {}^{\;\tilde \otimes n}$ in our setting.)
The main result of the present  paper is the following, which we prove in Section \ref{proof}.
\begin{theorem} \label{1}
Let $T$ be an $n$-component ribbon bottom tangle with $0$-framing.
Then we have $J_T\in  (\bar U_q^{ev})\;\hat  {}^{\;\hat  \otimes n}$.
\end{theorem}
 An $n$-component bottom tangle $T$ is called  a \textit{slice bottom tangle} if $T$ is  \textit{concordant} to the $n$-component trivial bottom tangle,
where the  trivial bottom tangle is the  bottom tangle taking the shape as $\cap\ldots \cap $ (see Section \ref{bot} for the definition of the   concordance of  bottom tangles).
The following is a  generalization of Conjecture \ref{Habico} and Theorem \ref{1}.
\begin{conjecture}
If an $n$-component bottom tangle $T$ is   concordant to a boundary bottom tangle (in particular, if $T$ is a slice bottom tangle),
then we have $J_T\in  (\bar U_q^{ev})\;\hat  {}^{\;\hat  \otimes n}$.
\end{conjecture}
\subsection{An application to the colored Jones polynomial.} 
Here, we give an application of Theorem \ref{1}.
We use the following $q$-integer  notations.
\begin{align*}
&\{i\}_q = q^i-1,\quad  \{i\}_{q,n} = \{i\}_q\{i-1\}_q\cdots \{i-n+1\}_q,\quad  \{n\}_q! = \{n\}_{q,n},
\\
&[i]_q = \{i\}_q/\{1\}_q,\quad  [n]_q! = [n]_q[n-1]_q\cdots [1]_q, \quad \begin{bmatrix} i \\ n \end{bmatrix} _q  = \{i\}_{q,n}/\{n\}_q!,
\end{align*}
for $i\in \mathbb{Z}, n\geq 0$.

For $l\geq 1$, let  $V_l$ denote the $l$-dimensional irreducible representation of $U_h$.
Let $\mathcal{R}$  denote the representation ring  of $U_h$ over  $\mathbb{Q}(v)$, i.e.,
$\mathcal{R}$ is the $\mathbb{Q}(v)$-algebra 
$$
\mathcal{R}= \Span _{\mathbb{Q}(v)}\{V_l \  | \ l\geq 1\}
$$ 
with the multiplication induced by the tensor product.

Habiro \cite{H2}  studied  the following  polynomials in $V_2$
\begin{align*}
\tilde P'_l&=\frac{v^{l}}{\{l\}_q!}\prod _{i=0}^{l-1}(V_2-v^{2i+1}-v^{-2i-1})\in  \mathcal{R},
\end{align*}
for $l\geq 0$, and  proved the following theorem.
\begin{theorem}[Habiro \cite{H2}]
Let $L$ be an $n$-component, algebraically-split, $0$-framed  link.
We have
\begin{align*}
J_{L; \tilde P'_{l_1},\ldots , \tilde P'_{l_n}}\in \frac{\{ 2l_j+1\}_{q, l_j+1}}{\{1\} _q}\mathbb{Z}[q,q^{-1}],
\end{align*} 
for $l_1,\ldots , l_n\geq 0$, where $j$ is a number such that   $l_j=\max\{l_i\}_{1\leq i\leq n}$.
\end{theorem}
Here $J_{L; \tilde P'_{l_1},\ldots , \tilde P'_{l_n}}$ is the colored Jones polynomial of $L$ associated to $\tilde P'_{l_1},\ldots , \tilde P'_{l_n}$ (see Section \ref{bottom}).
The above theorem is an important technical step in Habiro's construction of the unified Witten-Reshetikhin-Turaev invariants for integral homology spheres.
Habiro  \cite{H2} also  proved that Conjecture \ref{Habico} would imply the following Theorem \ref{4}, \textit{with a ribbon link replaced by a  boundary link.}
Thus,  Theorem \ref{4}  follows  from Theorem \ref{1} and Habiro's argument in \cite{H2}.
\begin{theorem}\label{4}
Let $L$ be an $n$-component ribbon link with $0$-framing.
We have
\begin{align*}
J_{L; \tilde P'_{l_1},\ldots , \tilde P'_{l_n}}\in \frac{\{ 2l_j+1\}_{q, l_j+1}}{\{1\} _q} I_{l_1}\cdots \hat I_{l_j}\cdots I_{l_n},
\end{align*} 
for $l_1,\ldots , l_n\geq 0$, where $j$  is a number such that   $l_j=\max\{l_i\}_{1\leq i\leq n}$.
Here,  for $l\geq 0$, $I_{l}$  is the ideal in $\mathbb{Z}[q,q^{-1}]$ generated by the elements $\{l-k\}_q!\{k\}_q!$ for 
$k=0,\ldots, l$, and    $\hat I_{l_j}$ denotes omission of $I_{l_j}$.  
\end{theorem}
\begin{remark}
For $m\geq 1$, let $\Phi _m(q)\in \mathbb{Z}[q]$ denote the $m$th cyclotomic polynomial.
It is not difficult to prove that $I_{l}, l\geq 0,$ is contained in the principle ideal generated by  $\prod_m\Phi _m(q)^{f(l,m)}$, where $
f(l,m)=\max \{0, \big\lfloor \frac{l+1}{m}\big\rfloor -1 \}$. 
Here for $r\in \mathbb{Q}$, we denote by $\lfloor r \rfloor$  the largest integer smaller than  or  equal to $r$.
\end{remark}
\begin{remark}
As we have mentioned, Eisermann \cite{Ei} proved that the Jones polynomial $V(L)\in \mathbb{Z}[v,v^{-1}]$ of an $n$-component
ribbon link $L$ is divisible by the Jones polynomial $V(O^n)=(v+v^{-1})^n$ of the  $n$-component unlink $O^n$.
This result does not follow  directly from Theorem \ref{4}.
However, we  give another proof of it in  \cite{master} by proving a refinement of Theorem \ref{1} involving
a subalgebra of $U_h^{\hat\otimes  n}$ smaller than $ (\bar U_q^{ev})\;\hat  {}^{\;\hat  \otimes n}$. We do not describe it in the present paper since the proof in \cite{master} is quite complicated and also since we expect further refinements.
\end{remark}

\subsection{Organization of the paper.}
The rest of the paper is organized as follows.
In Section \ref{bot}, we define  bottom tangles and ribbon bottom tangles. 
In Section \ref{pre}, we define the quantized enveloping algebra $U_h$, and its subalgebras.
In Section \ref{bottom},  we consider the universal $sl_2$ invariant of   bottom tangles and ribbon bottom tangles.
In Sections \ref{proof}, we    prove   Theorem \ref{1}. 
In Section \ref{exam}, we consider the cases of the Borromean tangle and  the Borromean rings.

\section{Bottom tangles and ribbon bottom tangles.}\label{bot}
In this section, we recall from \cite{H1} the notion of  \textit{bottom tangles}.
We also define the notion of  \textit{ribbon bottom tangles}, which is implicit in \cite{H1}. 
\subsection{Bottom tangles.}
An $n$-component \textit{bottom tangle} $T=T_1\cup \cdots \cup T_n $ is an oriented, ordered, framed tangle in a cube $[0,1]^3$
consisting of $n$ arcs $T_1,\ldots ,T_n,$ whose boundary points  are on the  bottom line $[0,1]\times \{\frac{1}{2}\}\times \{0\}$, such that for each $i=1,\ldots ,n$, the component $T_i$ runs 
from the $2i$th boundary point to the $(2i-1)$th  boundary point, where the boundary points are ordered by  the first coordinate.
As usual, we draw a bottom tangle as a diagram in a rectangle, see Figure \ref{fig:closure} (a),(b).
\begin{figure}
\centering
\includegraphics[width=13cm,clip]{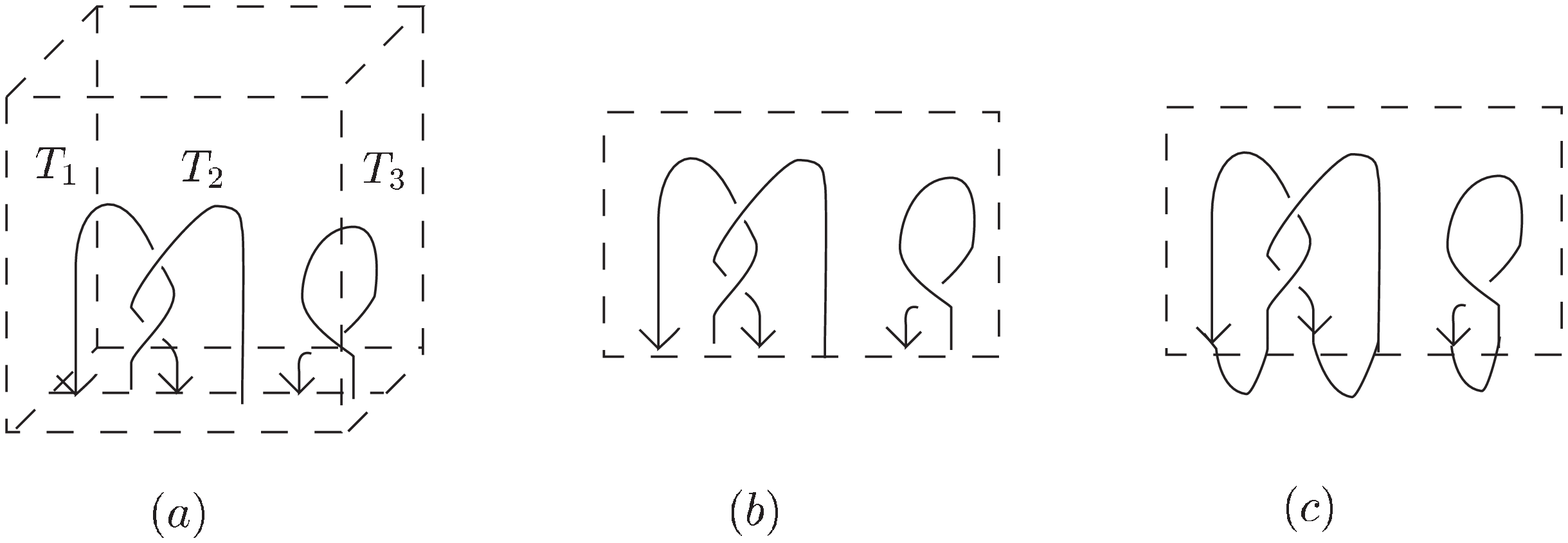}
\caption{ (a) A $3$-component bottom tangle $T=T_1\cup T_2\cup T_3$. (b) A diagram   of $T$ in a rectangle. (c) The closure  of $T$.}\label{fig:closure}
\end{figure}%
For each $n\geq 0$, let $BT_n$ denote the set of the isotopy classes of $n$-component
bottom tangles, and set $BT=\bigcup_{n\geq 0}BT_n$.

The \textit{closure}    of $T$ is the link  obtained from $T$  by pasting a ``$\cup$-shaped tangle'' to each component of  $T$,
as depicted in Figure \ref{fig:closure} (c). 
For any link $L$, there is a bottom tangle whose closure is isotopic to $L$. 

The \textit{linking matrix} $\mathrm{Lk}(T)$ of a bottom tangle $T=T_1\cup \cdots \cup T_n $ is defined as that of
the closure of  $T$. Thus, for $1\leq i\neq j\leq n$, the linking number of $T_i$ and $T_j$  is defined as the 
linking number of the corresponding components in the closure of $T$,
and, for $1\leq i\leq n,$ the framing of $T_i$ is defined as the framing of the closure of $T_i$. 

Two bottom tangle $T,T'\in BT_n$ are \textit{concordant} if there is an proper embedding;
\begin{align*}
f\colon\  \coprod ^n[0,1]\times [0,1] \hookrightarrow [0,1]^3\times [0,1],
\end{align*}
 such that
$f(\coprod ^n[0,1]\times \{0\})=T\times \{0\}$, $f(\coprod ^n[0,1]\times \{1\})=T'\times \{1\}$, and
\begin{align*}
f(\coprod ^n\partial [0,1]\times [0,1])= \partial T \times [0,1]=\partial T' \times [0,1].
\end{align*}
\subsection{Ribbon bottom tangles.}
 \begin{definition}
A   bottom tangle  $T\in BT$ is called a \textit{ribbon bottom tangle} if and only if the closure of $T$ is a  ribbon link.
\end{definition}
\textit{A system of ribbon disks} for an $n$-component  bottom tangle 
$T=T_1\cup \ldots\cup T_n$ is a  immersed surface with ribbon singularities in $[0,1]^3$ consisting of $n$ disks  bounded by the link
 $\tilde{T}= (T_1\cup \gamma_1)\cup \ldots \cup (T_n\cup \gamma_n) $,
where $\gamma_i \subset [0, 1]\times \{\frac{1}{2} \} \times \{ 0 \}$ is the line segment such that $\partial \gamma_i =\partial T_i$ for  $1\leq i\leq n$.
\begin{proposition}\label{rs}
A bottom tangle $T\in BT_n$ is a ribbon bottom tangle if and only if it  admits  a  system of ribbon disks.
\end{proposition}
\begin{proof}
Let $X\subset  S^3$ be a   system of  ribbon disks for the link  $\tilde T$.
Up to isotopy in $S^3$ fixed on the link  $\tilde T$,  we can assume that $X\subset [0,1]^2\times [-1,1]$.
If we admit introducing new ribbon singularities, we can transform $X$ into a system of ribbon disks for the bottom tangle $T$
 by pulling the segment part  $\gamma_i \subset [0, 1]\times \{\frac{1}{2} \} \times \{ 0\}$  straight down to the $[0,1]\times \{\frac{1}{2} \} \times \{ -1\}$, 
 and transforming  $[0,1]^2\times [-1,1]$ into $[0,1]^3$ by isotopy of $S^3$. For example,  see Figure \ref{fig:singularity2}.
\end{proof}

\begin{figure}

\centering
\includegraphics[width=13cm,clip]{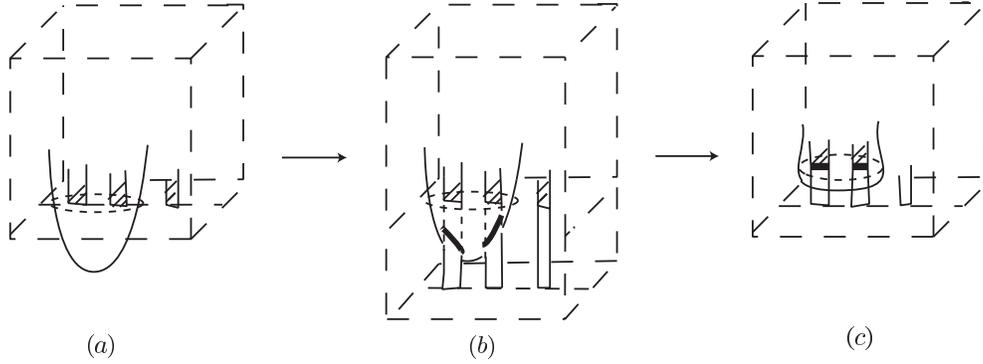}
\caption{(a) A system of  ribbon disks for the closure link $\tilde T$. 
(b) A system of ribbon disks for a link isotopic to $\tilde T$.
(c) A system of ribbon disks for bottom tangle $T$.}\label{fig:singularity2}
\end{figure}
\section{The quantized enveloping algebra $U_h$ and its subalgebras.}\label{pre} 
We mostly follow the notations in \cite{H2}.
\subsection{The quantized enveloping algebra $U_h$.}\label{env}
Recall that $v=\exp\frac{h}{2}$, and $q=v^2.$
We denote by  $U_h$ the $h$-adically complete $\mathbb{Q}[[h]]$-algebra,
topologically generated by the elements $H, E,$ and $F$, satisfying the relations
\begin{align*}
HE-EH=2E, \quad HF-FH=-2F, \quad EF-FE=\frac{K-K^{-1}}{v-v^{-1}},
\end{align*}
where we set 
\begin{align*}
K=v^H=\exp\frac{hH}{2}.
\end{align*}

We equip $U_h$  with a topological $\mathbb{Z}$-graded algebra structure with  $\deg F=-1$,  $\deg E=1$, and   $\deg H=0$.
For a homogeneous element $x$ of $U_h$, the degree of $x$ is denoted by $|x|$.

There is a unique  complete ribbon Hopf algebra  structure  on  $U_h$   such that
\begin{align*}
\Delta (H)&=H\otimes 1+1\otimes H, \quad  \varepsilon (H)=0, \quad S(H)=-H,
\\
\Delta (E)&=E\otimes 1+K\otimes E, \quad \varepsilon (E)=0, \quad  S(E)=-K^{-1}E,
\\
\Delta (F)&=F\otimes K^{-1}+1\otimes F, \quad  \varepsilon (F)=0, \quad  S(F)=-FK.
\end{align*}
The universal $R$-matrix and its inverse are given by
\begin{align}
R&=D\bigg(\sum_{n\geq 0}v^{\frac{1}{2}n(n-1)}\frac{(v-v^{-1})^n}{[n]!}F^n\otimes E^n\bigg),\label{rm1}
\\
R^{-1}&=D^{-1}\bigg(\sum_{n\geq 0}(-1)^nv^{-\frac{1}{2}n(n-1)}\frac{(v-v^{-1})^n}{[n]!}F^nK^n\otimes K^{-n}E^n\bigg),\label{rm2}
\end{align} 
where $D=v^{\frac{1}{2}H\otimes H} =\exp \big(\frac{h}{4}H\otimes H\big)\in U_h^{\hat {\otimes }2}$.
The ribbon element and its inverse are given by
\begin{align*}
r=\sum \bar {\alpha }K^{-1}\bar {\beta}=\sum \bar {\beta} K\bar {\alpha}, \quad  r^{-1}=\sum \alpha K\beta =\sum\beta K^{-1}\alpha,
\end{align*}
where  $R=\sum \alpha \otimes \beta $, and  $R^{-1}=(S\otimes 1)R=\sum \bar {\alpha} \otimes \bar {\beta }$.

We use  notations   $D=\sum D^+_{[1]}\otimes D^+_{[2]}$, and  $D^{-1}=\sum D^-_{[1]}\otimes D^-_{[2]}$.
We shall use the following formulas.
\begin{align}
&\sum D^+_{[2]}\otimes D^+_{[1]}=D, \quad   (\Delta \otimes 1)D=D_{13}D_{23},\label{d1}
\\
&(\varepsilon \otimes 1)(D)=1, \quad (1\otimes S)D=(S\otimes 1)D=D^{-1},
\\
&D(1\otimes x)=(K^{|x|}\otimes x)D,\label{exD}
\end{align}
where $D_{13}=\sum D^+_{[1]}\otimes 1\otimes D^+_{[2]}$, $D_{23}=1\otimes D$, and  $x$ is a homogeneous element of $U_h$.

\subsection{Subalgebras $U_{\mathbb{Z},q}$ and  $U_{\mathbb{Z},q}^{ev}$ of  $U_h$.} \label{uqb1}
For $i\in \mathbb{Z}, n\geq 0$, set
\begin{align*}
[i] = \frac{v^i-v^{-i}}{v-v^{-1}},\quad  [n]! = [n][n-1]\cdots [1].
\end{align*}
Let $U_{\mathbb{Z}}$ denote   Lusztig's  integral form of $U_h$ (cf. \cite{L}), which  is defined to be the $\mathbb{Z}[v,v^{-1}]$-subalgebra of $U_h$ generated by $K,K^{-1},$ 
 $E^{(n)}=E^n/[n]!$, and  $F^{(n)}=F^n/[n]!$ for $n\geq 1$.

Set
\begin{align*}
&\tilde E^{(n)}=(v^{-1}E)^n/[n]_q!=v^{-\frac{1}{2}n(n+1)}E^{(n)}, \\ &\tilde {F}^{(n)}=F^nK^n/[n]_q!=v^{-\frac{1}{2}n(n-1)}F^{(n)}K^n,
\end{align*}
for $n\geq 0$.
Let $U_{\mathbb{Z}, q}$ denote the $\mathbb{Z}[q,q^{-1}]$-subalgebra of $U_{\mathbb{Z}}$ generated by
$K,K^{-1}, \tilde E^{(n)}$, and  $\tilde F^{(n)}$ for $n\geq 1$.
Note that 
\begin{align*}
U_{\mathbb{Z}}=U_{\mathbb{Z}, q}\otimes _{\mathbb{Z}[q,q^{-1}]}\mathbb{Z}[v,v^{-1}].
\end{align*}
Let $U_{\mathbb{Z}, q}^{ev}$ denote the $\mathbb{Z}[q,q^{-1}]$-subalgebra of $U_{\mathbb{Z}, q}$ generated by
$K^2,K^{-2}, \tilde E^{(n)}$, and  $\tilde F^{(n)}$ for $n\geq 1$.
$U_{\mathbb{Z}, q}$ is equipped with a $(\mathbb{Z}/2\mathbb{Z})$-graded $\mathbb{Z}[q,q^{-1}]$-algebra
structure
$$
U_{\mathbb{Z}, q}=U_{\mathbb{Z}, q}^{ev}\oplus KU_{\mathbb{Z}, q}^{ev}.
$$

There is a Hopf $\mathbb{Z}[q,q^{-1}]$-algebra structure  on $U_{\mathbb{Z}, q}$ inherited from $U_h$ such that
\begin{align}
\Delta (K^i)=K^i\otimes K^i,& \quad S^{\pm  1}(K^i)=K^{-i},
\\
\Delta (\tilde E^{(n)})=\sum_{j=0}^n \label{De1}
\tilde E^{(n-j)}K^j\otimes \tilde E^{(j)},& \quad
\Delta (\tilde {F}^{(n)})=\sum_{j=0}^n 
\tilde {F}^{(n-j)}K^j\otimes \tilde {F}^{(j)}, \\
S^{\pm 1}(\tilde E^{(n)})=(-1)^nq^{\frac{1}{2}n(n\mp 1)}K^{-n}\tilde E^{(n)},& \quad
S^{\pm 1}(\tilde {F}^{(n)})=(-1)^nq^{-\frac{1}{2}n(n\mp 1)}K^{-n}\tilde {F}^{(n)},
\\ 
\varepsilon (K^i)=1, \quad \varepsilon (\tilde {E}^{(n)})&=\varepsilon (\tilde {F}^{(n)})=\delta _{n,0},
\end{align}
for $i\in \mathbb{Z},  n\geq 0$.
\subsection{Subalgebras $\bar U_q$ and  $\bar U_q^{ev}$ of  $U_h$.} \label{uqb2}
Let $\bar U$ denote   the $\mathbb{Z}[v,v^{-1}]$-subalgebra of $U_h$ 
generated by the elements $K,K^{-1}$, $(v-v^{-1})E$, and $(v-v^{-1})F$ (cf. \cite{De}).

Set  
$$e=v^{-1}(q-1)E, \quad f=(q-1)FK. $$
Let $\bar {U}_q$ denote the $\mathbb{Z}[q,q^{-1}]$-subalgebra of $U_{\mathbb{Z},q}$ generated by
the elements $K,K^{-1},e$ and $f$.
Note that 
\begin{align*}
\bar U=\bar U_{ q}\otimes _{\mathbb{Z}[q,q^{-1}]}\mathbb{Z}[v,v^{-1}].
\end{align*}
Let $\bar {U}_q^{ev}$ denote the $\mathbb{Z}[q,q^{-1}]$-subalgebra of $U_{\mathbb{Z},q}^{ev}$ generated by
the elements $K^2,K^{-2},e$ and $f$. We have
\begin{align*}
\bar {U}_q^{ev}=\bar {U}_q\cap U_{\mathbb{Z},q}^{ev}, \quad \bar {U}_q=\bar {U}_q^{ev}\oplus K\bar {U}_q^{ev}.
\end{align*}

There is  a   Hopf $\mathbb{Z}[q,q^{-1}]$-algebra structure  on $\bar {U}_q$ inherited from $U_h$ such that
\begin{align}
\Delta (e^n)=\sum_{j=0}^n 
\begin{bmatrix}n \\j \end{bmatrix}_q
e^{n-j}K^j\otimes e^j,& \quad \Delta (f^n)
=\sum_{j=0}^n\begin{bmatrix}n \\j \end{bmatrix}_q q^{-j(n-j)} f^{n-j}K^j\otimes f^j, 
\\
S^{\pm 1}(e^n)=(-1)^nq^{\frac{1}{2}n(n\mp 1)}K^{-n}e^n,& \quad S^{\pm 1}(f^n)=(-1)^nq^{-\frac{1}{2}n(n\mp 1)}K^{-n}f^n, 
\\
 \varepsilon (e^n)&=\varepsilon (f^n)=\delta _{n,0},
\end{align}
for $n\geq 0$.

We have
\begin{align}
e^mf^n=\sum_{p=0}^{\min (m,n)} q^{\frac{1}{2}p(p+1)-nm}& \{p\}_q!
\begin{bmatrix}m \\p \end{bmatrix}_q
\begin{bmatrix}n\\p \end{bmatrix}_q
f^{n-p}\{ H-m-n+2p \}_{q,p} e^{m-p},\label{B2}
\end{align}
for $m,n\geq 0$.
Here, for $i\in \mathbb{Z}$ and $p\geq 0$, we set
$$
\{H+i\}_{q,p}=\{H+i\}_q\{H+i-1\}_q\cdots \{H+i-p+1\}_q,
$$
where
$$
\{H+j\}_q=q^{H+j}-1=q^jK^2-1,
$$
for $j\in \mathbb{Z}$.

The following lemma, which is a $\mathbb{Z}[q,q^{-1}]$-version of a well known result for $\bar U$ by De Concini and Procesi \cite{De}, can be proved by using the formula  (\ref{B2}).
\begin{lemma}\label{F2}
 $ \bar {U}_q$ (resp. $ \bar {U}_q^{ev}$) is freely $\mathbb{Z}[q,q^{-1}]$-spanned by the elements $f^iK^j e^k$
 (resp. $f^iK^{2j} e^k$)  with $i,k\geq 0$ and $j\in \mathbb{Z}$. 
\end{lemma}

\subsection{Adjoint action.}
 We use  the left adjoint action of $U_h$ defined by
\begin{align*}
\ad(a\otimes b):=\sum a'bS(a''), 
\end{align*}
where  $\Delta (a)=\sum a'\otimes a''$. 
We also use the notation $a\triangleright b:=\ad(a\otimes b).$

The following proposition is suggested by Habiro.
In fact, Habiro and Le \cite{H4} prove a generalization of a $\mathbb{Z}[v,v^{-1}]$-version of the following proposition with $i=0$
to quantized enveloping algebras for all simple Lie algebras.
\begin{proposition} \label{Habi}
For $i=0,1$, we have
\begin{align*}
 U_{\mathbb{Z}, q}\triangleright K^i\bar {U}_q^{ev}\subset K^i\bar {U}_q^{ev}.
\end{align*}
\end{proposition}
\begin{proof}
In view of Lemma \ref{F2}, it is enough to prove that $x\triangleright f^{i_1}K^{i_2}e^{i_3}\in K^{i_2}\bar {U}_q^{ev}$ for every $x\in \{K,K^{-1}, \tilde E^{(n)},\tilde F^{(n)} \ | \ n\geq 0\}$ and  $i_1,i_3\geq 0$, $i_2\in \mathbb{Z}$. By  computation, we have 
\begin{align}
K^{\pm  1}&\triangleright  f^{i_1}K^{i_2}e^{i_3}=q^{\pm  (i_3-i_1)}f^{i_1}K^{i_2}e^{i_3}, \label{K} 
\\
\begin{split}
\tilde E^{(n)}&\triangleright  f^{i_1}K^{i_2}e^{i_3}
\\=&\sum_{p=0}^{\min( i_1, n )}(-1)^n q^{\frac{1}{2}p(p+1)-n(i_1+i_2)+2i_2p} \label{KE}
\begin{bmatrix}i_1 \\p \end{bmatrix}_q
f^{i_1-p}K^{i_2}g(i_1,i_2,i_3,n,p)e^{i_3+n-p},
\end{split}
\\
\begin{split}\label{KF}
\tilde F^{(n)}&\triangleright  f^{i_1}K^{i_2}e^{i_3}
\\ =&\sum_{p=0}^{\min( i_3, n )} q^{\frac{1}{2}p(p+1)-n(i_1+i_2)+2i_2p}
\begin{bmatrix}i_3 \\p \end{bmatrix}_q f^{n+i_1-p}K^{i_2}g(i_3,i_2,i_1,n,p)e^{i_3-p},
\end{split}
\end{align}
where
\begin{align*}
g(i_1,i_2,i_3,n,p)=& \sum_{s=0}^{p}(-1)^sq^{\frac{1}{2}s(s+1)-s(n-p+i_1)}\begin{bmatrix}p \\s \end{bmatrix}_q 
\begin{bmatrix}n-p+i_2+i_3+s-1 \\n-p \end{bmatrix}_qK^{2s}.
\end{align*}
The right hand sides of (\ref{K})--(\ref{KF}) are all contained  in $K^{i_2}\bar {U}_q^{ev}$, hence we have the assertion. 
\end{proof}

\section{The universal $sl_2$ invariant of bottom tangles.}\label{bottom}
In this section,  we define   the universal $sl_2$ invariant of  bottom tangles \cite{H1}, and study the values of it.
Then we discuss the case of ribbon bottom tangles.
\subsection{Decorated diagrams.}
We use diagrams of tangles obtained from copies of the fundamental tangles, as depicted in Figure \ref{fig:fundamental},
by pasting horizontally and vertically. 
\begin{figure}
\centering
\includegraphics[width=9cm,clip]{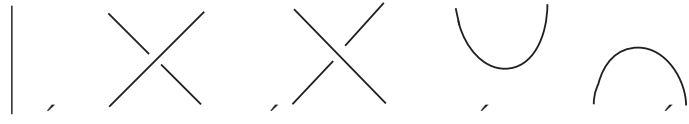}
\caption{Fundamental tangles. The orientations of the strands are arbitrary. }\label{fig:fundamental}
\end{figure}%
 A \textit{decorated diagram} of  a bottom tangle $T\in BT$ is a diagram $P$ of $T$
together with finitely many dots on strands, each labeled by an element of  $U_h$.
We also  allow  pairs of dots, each  connected  by  an oriented dashed line which is labeled by an element of $U_h^{\hat \otimes 2}$
so that the first  tensorand is attached to the start point of the line, and the second tensorand to the end point,  see Figure \ref{fig:deco} (a).
\begin{figure}
\centering
\includegraphics[width=10cm,clip]{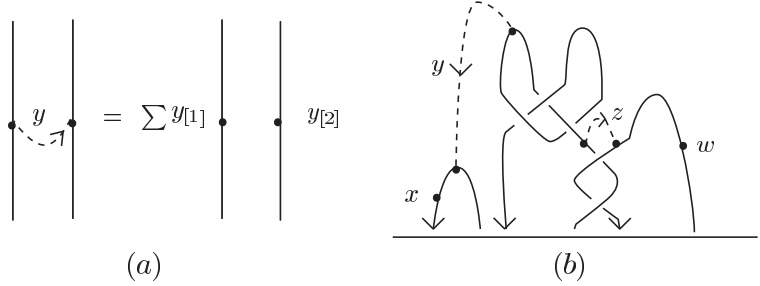}
\caption{(a) How to label an element $y=\sum y_{[1]}\otimes y_{[2]}$ to the  connected  dots. (b) A decorated diagram $P$.}\label{fig:deco}
\end{figure}
 If the element  $y\in U_h^{\hat \otimes 2}$ on it  is symmetric, we do not have to specify the orientation of a dashed line.

For every decorated diagram $P$ for an $n$-component bottom tangle $T=T_1\cup\cdots \cup T_n \in BT_n$,
we define an element  $J(P)\in U_h^{\hat {\otimes }n}$ as follows.
The $i$th component of $J(P)$ is defined  to be the product 
of the elements put on the component corresponding to $T_i$, where the elements are read off along each component
reversing the orientation of $P$, and   written from left to  right. 
For example, for the decorated diagram $P$ depicted in  Figure \ref{fig:deco} (b), we have
\begin{align*}
J(P)=\sum xy_{[2]}\otimes y_{[1]}z_{[1]}\otimes z_{[2]}w,
\end{align*}
where $y=\sum y_{[1]}\otimes y_{[2]}$ and $z=\sum z_{[1]}\otimes z_{[2]}$.
In what follows, we sometimes  identify a decorated diagram and its image by $J$.
For example, the picture depicted in  Figure \ref{fig:exD} represents the formula (\ref{exD}).
\begin{figure}
\centering
\includegraphics[width=7cm,clip]{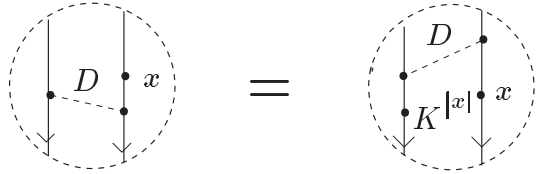}
\caption{A graphical version of  (\ref{exD}).
By the two pictures above, we mean two decorated diagrams of a  bottom tangle which are identical outside the dotted circles.}\label{fig:exD}
\end{figure}
\subsection{The universal $sl_2$ invariant of bottom tangles.}\label{bottom inv}
For  $T=T_1\cup \cdots \cup T_n\in BT_n$, we define the universal $sl_2 $ invariant $J_T\in U_h^{\hat {\otimes }n}$ of $T$ as follows.
We choose a  diagram $P$  of $T$. We denote by $C(P)$ the set of the crossings of $P.$
We call a map 
\begin{align*}
s\colon\; C(P) \ \ \rightarrow \ \ \{0,1,2,\ldots\}
\end{align*}
a \textit{state}. We denote by  $\mathcal{S}(P)$ the set of states for $P$.
For each state $s\in \mathcal{S}(P)$, we define a decorated diagram $(P,s)$ (by abusing the notation) as follows.

We  rewrite the $R$-matrix (\ref{rm1}) and its inverse (\ref{rm2}) as 
\begin{align}
&R^{\pm1}= D^{\pm1}\sum_{n\geq 0}R_n^{\pm}, \\
R_n^+=q^{\frac{1}{2}n(n-1)}
&\tilde {F}^{(n)}K^{-n}\otimes e^n, \quad
R_n^{-}=(-1)^{n}\tilde {F}^{(n)}\otimes K^{-n}e^n.
\end{align}  
We use the notations
$R_n^+=\sum R^+_{n[1]}\otimes R^+_{n[2]}$ and 
$R^-_n=\sum R^-_{n[1]}\otimes R^-_{n[2]}.$

 For each fundamental tangle in $P$, we attach  elements 
following the rule described in Figure \ref{fig:cross}, where  ``$S'$'' should be replaced with id if 
the string is oriented downward, and with $S$ otherwise, see Figure \ref{fig:S'}.
Thus we have an element $J(P,s)\in U_h^{\hat \otimes n}$ as the image of the decorated diagram $(P,s)$ by $J$.

\begin{figure}
\centering
\includegraphics[width=10cm,clip]{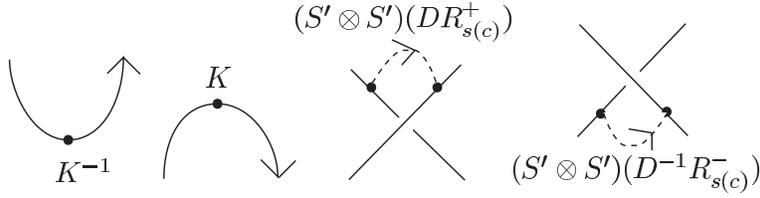}
\caption{How to place elements  on the fundamental tangles.}\label{fig:cross}
\end{figure}
\begin{figure}[t]
\centering
\includegraphics[width=7cm,clip]{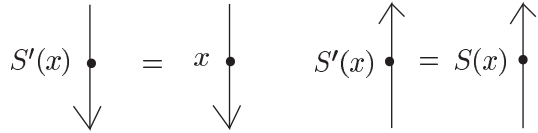}
\caption{The definition of $S'$.}\label{fig:S'}
\end{figure}%
Set 
$$
J_T=\sum_{s\in \mathcal{S}(P)}J(P, s).
$$
As is well known \cite{O},  $J_T$ does not depend on the choice of the diagram $P$, and defines an isotopy invariant of bottom tangles.

For example, let us compute the universal $sl_2$ invariant $J_C$ of a bottom tangle $C$ with a diagram $P$  as depicted in  Figure \ref{fig:base} $(a)$,
where  $c_1$ (resp. $c_2$) denotes the upper (resp. lower)  crossing  of $P$.
The decorated diagram $(P,s)$ for the state $s\in \mathcal{S}(P)$   is depicted in  Figure \ref{fig:base} $(b)$, where we set $m=s(c_1), n=s(c_2)$.
\begin{figure}[t]
\centering
\includegraphics[width=8cm,clip]{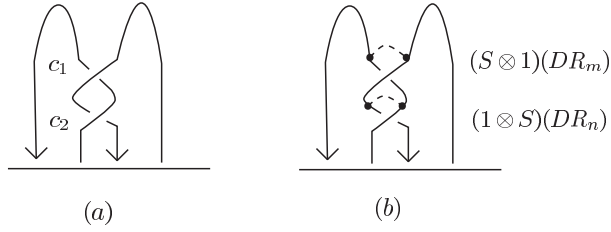}
\caption{$(a)$ A  diagram $P$ of $C\in BT_2$. $(b)$ The decorated  diagram $(P,s)$.
} \label{fig:base}
\end{figure}
We have
\begin{align*}
J_{C}&=\sum_{s\in \mathcal{S}(P)}J(P,s)
\\
&=\sum_{s\in \mathcal{S}(P)}\sum S(D^+_{[1]}R^+_{m[1]})S(D'^+_{[2]}R^+_{n[2]})\otimes D'^+_{[1]}R^+_{n[1]}D^+_{[2]}R^+_{m[2]}
\\
&=\sum_{m,n\geq 0} (-1)^{m+n}q^{-n+2mn}D^{-2}(\tilde F^{(m)}K^{-2n}e^n\otimes \tilde F^{(n)}K^{-2m}e^m).
\end{align*}
where  $D^{\pm 1}=\sum D^{\pm}_{[1]}\otimes D^{\pm}_{[2]} =\sum D'^{\pm}_{[1]}\otimes D'^{\pm}_{[2]}$.
\subsection{The colored Jones polynomial.}
If $V$ is a finite dimensional representation of $U_h$, then the quantum trace $\tr_q^V(x)$ in $V$
of an element $x\in U_h$ is defined by   
\begin{align*}
\tr_q^V(x)=\tr^V(\rho_V(K^{-1}x))\in \mathbb{Q}[[h]],
\end{align*}
where $\rho_V\colon\; U_h\rightarrow \End(V)$ denotes the left action of $U_h$ on $V$, and $\tr^V\colon\; \End (V)\rightarrow \mathbb{Q}[[h]]$
denotes the trace in $V$.
For every element $y=\sum_n a_nV_n\in \mathcal{R}$, $a_n\in \mathbb{Q}(v)$, 
we set 
\begin{align*}
\tr_q^{y}(x)=\sum_na_n \tr_q^{V_n}(x)\in \mathbb{Q}(v)
\end{align*}
for $x\in U_h$.

The universal $sl_2$ invariant of bottom tangles has  a universality property to the colored Jones polynomials of links as the following. 
\begin{proposition}[Habiro \cite{H2}]\label{r2}
Let $L=L_1\cup \cdots \cup  L_n$ be an $n$-component, ordered,  oriented, framed link in $S^3$.
Choose an $n$-component bottom tangle $T$ whose closure  is isotopic to $L$.
For  $y_1,\ldots, y_n \in \mathcal{R}$, the colored Jones polynomial 
$J_{L;y_1, \ldots ,y_n}$ of $L$ can be obtained from $J_T$ by 
\begin{align*}
J_{L;y_1,\ldots,y_n}=(\tr_q^{y_1}\otimes \cdots \otimes \tr_q^{y_n})(J_T).
\end{align*}
\end{proposition}
\subsection{Values of the universal $sl_2$ invariant of bottom tangles.}
In this subsection we consider the value of $J(P,s)$ for a decorated diagram $(P,s)$.
Let us prepare some notations.

For $n\geq 1, 1\leq i\leq n$,  and  for  $X\in U_h$, we define
$X_i\in U_h^{\hat {\otimes n }}$  by
$$X_i=1\otimes \cdots \otimes X
\otimes \cdots \otimes 1,
$$
 where $X$ is at the $i$th position.
 
For $1\leq i ,j\leq n$, and  for  $Y=\sum y_1\otimes y_2\in U_h^{\hat {\otimes }2}$, we define  $Y_{ij} \in U_h^{\hat {\otimes n }}$ by
$$
Y_{ij}=\sum (y_1)_i(y_2)_j.
$$
 For every symmetric integer  matrix $M=(m_{ij})_{1\leq i,j \leq n}$ of size $n\geq 1$,
we define  two invertible  elements $D^M, \tilde {D}^M\in U_h^{\hat {\otimes }n}$ by
\begin{align*}
D^M&=\prod _{1\leq i,j\leq n}D_{ij}^{m_{ij}}=\prod _{1\leq i<j\leq n}D_{ij}^{2m_{ij}}\prod _{1\leq i\leq n}(v^{H^2/2})_i^{m_{ii}},
\\
\tilde {D}^{M}&=D^{M}\prod _{1\leq i\leq n} K_i^{m_{ii}}=\prod _{1\leq i< j\leq n}D_{ij}^{2m_{ij}}\prod _{1\leq i\leq n}(v^{H^2/2}K)_i^{m_{ii}}.
\end{align*}
Later, we shall use  the following proposition.
\begin{proposition}\label{WW}
Let  $T=T_1\cup \cdots \cup T_n$ be  an $n$-component  bottom tangle.
For every diagram $P$ of $T$ and every  state $s\in \mathcal{S}(P)$, we have
\begin{align*}
J(P,s)\in \tilde D^{\mathrm{Lk}(T)}({U}_{\mathbb{Z},q}^{ev})^{\otimes n}.
\end{align*}
\end{proposition}
Before proving Proposition \ref{WW}, we  modify the dots of the decorated diagram  $(P,s)$.
Then  we define  three decorated diagrams $(P,s)^{\circ }, (P,s)^{\bullet },$ and  $(P,s)^{\diamond  }$, which we use in the proof of Proposition \ref{WW}.

In what follows, we can work up to the equivalence relation $\sim $ on $({U}_{\mathbb{Z},q})^{\otimes n}$ 
generated by multiplication on any tensorands by  $\pm q^{j}, K^{2j} ( j\in \mathbb{Z})$.
The modification process goes  as follows.  Let $c$ be a crossing of  $(P,s)$ with strands oriented downward,  and  set $m=s(c)$. As depicted in Figure \ref{fig:cross3}, we replace the two dots labeled by $D^{\pm 1}R_{m}^{\pm}$ with two  black  dots  labeled by 
  $D^{\pm 1}$   and  two white dots  labeled by  $R_{m}^{\pm}$.
\begin{figure}
\centering
\includegraphics[width=11cm,clip]{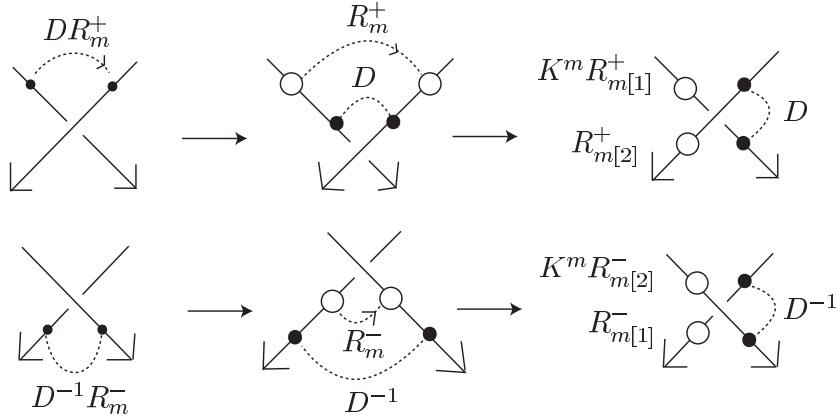}
\caption{The modification process of  $(P,s)$ on positive and negative crossings.}\label{fig:cross3}
\end{figure}%
Then  we slide  the black (resp. white) dots to the right hand side (resp. the left hand side)  of the crossings, and 
put the  produced element $K^{m}$ into the same dot of $R_m^{\pm}$.
Here the transformation  follows from  the formulas 
\begin{align*}
DR^+_m&=\sum D_{[1]}R^+_{m[1]}\otimes D_{[2]}R^+_{m[2]}
\\
&=\sum D_{[1]}K^{m}R^+_{m[1]}\otimes R^+_{m[2]} D_{[2]}.
\end{align*}
and
\begin{align*}
D^{-1}R^-_m&=\sum D^-_{[1]}R^-_{m[1]}\otimes D^-_{[2]}R^-_{m[2]}
\\
&=\sum R^-_{m[1]} D^-_{[1]}\otimes  D^-_{[2]} K^mR^-_{m[2]}.
\end{align*} 
Note that
\begin{align}
K^{m}R^+_{m[1]}\otimes  R^+_{m[2]}&\sim \label{rem1}
\tilde {F}^{(m)}\otimes e^m,
\\
R^-_{m[1]}\otimes K^mR^-_{m[2]}&\sim \tilde {F}^{(m)}\otimes e^m.\label{rem2}
\end{align}  
Similarly, we modify the dots on the other crossings as depicted in Figure \ref{fig:crosss2}. 
\begin{figure}
\centering
\includegraphics[width=11cm,clip]{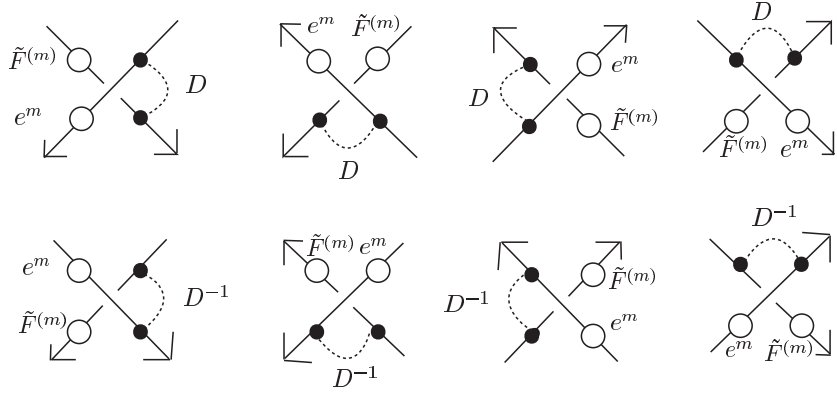}
\caption{Crossings of the decorated diagram $(P,s)$ after the modification.}\label{fig:crosss2}
\end{figure}%
We have completed the modification. By abusing the notation, we denote by $(P,s)$  the decorated diagram obtained from the modification.

We define  the  decorated diagrams $(P,s)^{\circ }$, $(P,s)^{\bullet },$ and  $(P,s)^{\diamond  }$  as follows.
\begin{itemize}
\item[(1)]
Let $(P,s)^{\circ }$ denote the  diagram $P$ together with the white dots  on crossings of  $(P,s)$.
Note that
\begin{align}
J(P,s)^{\circ }\in (U_{\mathbb{Z},q}^{ev})^{\otimes n}.\label{noten}
\end{align}
\end{itemize}
Let  $\vec \cap $, $\vec \cup $ and $\cap $ denote the fundamental tangles defined by
\begin{center}
\includegraphics[width=6cm,clip]{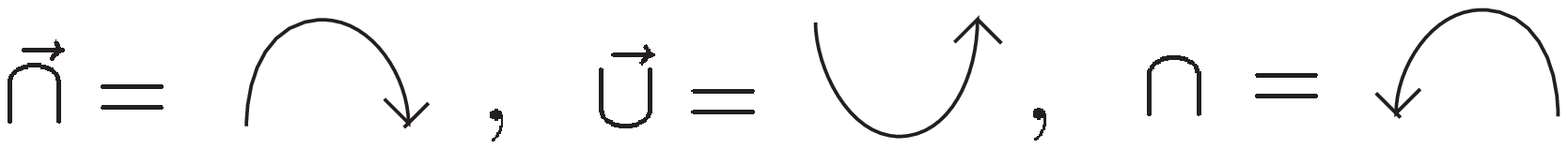}.
\end{center}
\begin{itemize}
\item[(2)]
Let $(P,s)^{\bullet }$ denote the diagram  $P$ with the   black dots labeled by $D^{\pm 1}$ on crossings of $(P,s)$, and  dots on $\vec \cap$ and $ \vec \cup $ of $(P,s)$.
\item[(3)]
For $i=1,\ldots,n,$ let $P_i$ denote the part of $P$ corresponding to $T_i$.
We call  the $2i$th (resp. $(2i-1)$th) boundary point of $P$ the \textit{start point} (resp.  \textit{end point}) of $P_i$.
 On $(P,s)$,  we slide all  white dots to the start points of the strands of $P$.
When we slide a  white dot through a dot on $\vec \cap$ or $ \vec \cup $, a scalar $q^j ( j\in \mathbb{Z})$  appears,
which we can ignore.
When we slide a  white dot through a dot labeled by $D^{\pm}$, a power of $K$   appears, see Figure \ref{fig:D2}.
We attach such element to a new white  diamond.
\begin{figure}
\centering
\includegraphics[width=7cm,clip]{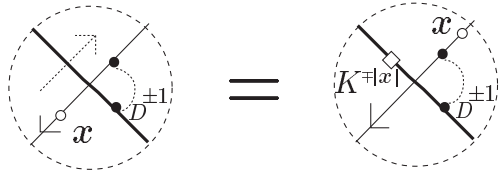}
\caption{The picture  when we slide a homogeneous  $x$ through a dot  labeled by $D^{\pm 1}$. 
This is essentially the same with the picture in Figure \ref{fig:exD}. }\label{fig:D2}
\end{figure}%
Let  $(P_i,s)^{\diamondsuit  }$ be the  diagram $P_i$ with the white diamonds  on $P_i$. Set $$J(P,s)^{\diamondsuit }=J(P_1,s )^{\diamondsuit }\otimes \cdots \otimes J(P_n,s )^{\diamondsuit }.$$
\end{itemize}
\begin{figure}
\centering
\includegraphics[width=10cm,clip]{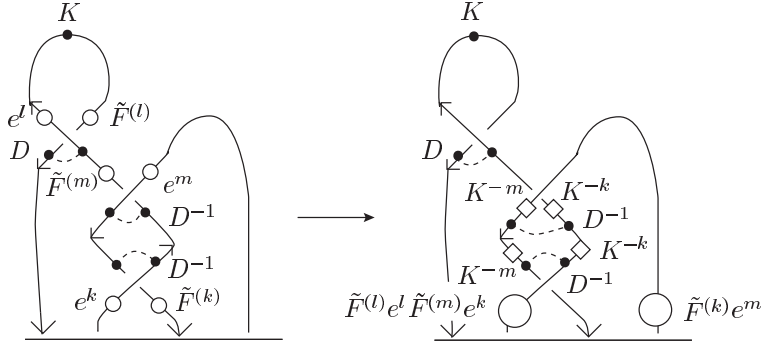}
\caption{The sliding process for a  decorated diagram $(P,s)$, where we set $s(c_1)=l, s(c_2)=m$, and $ s(c_3)=k$ for the upper, the middle, and the lower crossings $c_1,c_2,$ and $c_3,$ respectively.
We work up to multiplication by  $\pm q^{j}, K^{2j} (j\in \mathbb{Z})$. }\label{fig:LLink222}
\end{figure}%
For example, for the decorated diagram $(P,s)$ in Figure \ref{fig:LLink222}, we have 
\begin{align*}
&\mathrm{Lk}(T)=
\begin{pmatrix}
1 & -1 \\
-1 & 0
\end{pmatrix},
\\
&\tilde D^{\mathrm{Lk}(T)}=D^{-2}(v^{H^2 /2}K\otimes 1),
\\
&J(P,s)^{\circ }\sim \tilde F^{(l)} e^{l} \tilde F^{(m)} e^k\otimes  \tilde{F}^{(k)}e^m,
\\
&J(P,s )^{\bullet }\sim D^{-2}(v^{H^2/2}K\otimes 1),
\\
&J(P_1,s )^{\diamondsuit }\sim K^{-2k}\sim 1,
\\
&J(P_2,s)^{\diamondsuit }\sim K^{-2m}\sim 1.
\end{align*}

We reduce Proposition \ref{WW} to the following two lemmas.
\begin{lemma}\label{Cap}
For every  diagram $P$ of a bottom tangle $K\in BT_1$ with framing $r(K)\in \mathbb{Z}$, 
let $u(P)\in \mathbb{Z}_{\geq 0}$ be the total number of  the copies of  $\vec {\cap }$ and 
$\vec {\cup }$  which are contained in $P$.
Then,  the sum $u(P)+r(K)$ is even.
\end{lemma}
\begin{proof}
Note that the parity of $u(P)+r(K)$  does not change  by the Reidemeister moves RI, RII, RIII, and   crossing changes as  depicted in Figure \ref{fig:Ri}.
\begin{figure}
\centering
\includegraphics[width=9cm,clip]{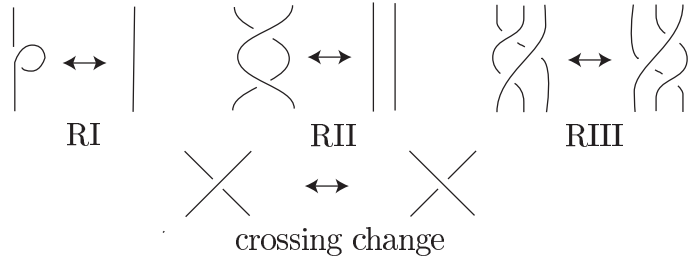}
\caption{The Reidemeister moves RI, RII, RIII, and  the crossing change.}\label{fig:Ri}
\end{figure}%
Since $P$ is equal to the bottom tangle $\cap $ up to those moves, we have
 \begin{align*}
u(P)+r(K)\equiv u(\cap )+r(\cap )=0 \ \  \pmod 2.
\end{align*}
This completes the proof.
\end{proof}
Let ${U}_h^0$  denote the $\mathbb{Q}[[h]]$-subalgebra of ${U}_h$ generated by $K,K^{-1}$.
Set
\begin{align*}
 \quad \bar {U}_q^{ev0}=\bar {U}_{q}^{ev}\cap {U}_h^0,
\end{align*}
which is  the $\mathbb{Z}[q,q^{-1}]$-subalgebra of $\bar {U}_{q}^{ev}$ generated by  $K^2,K^{-2}$. 
\begin{lemma}\label{bu}
We have
\begin{align*}
J(P,s )^{\bullet }\in \tilde  D^{\mathrm{Lk}(T)}(\bar U_q^{ev0})^{\otimes n}.
\end{align*}
\end{lemma}
\begin{proof}
For each $i=1,\ldots , n$, we denote by  $\kappa _i$  the product of  the $K^{\pm 1}$s on the copies of $\vec {\cap }$ and $\vec {\cup }$ of $P_i$. 
We have
\begin{align*}
J(P,s)^{\bullet }=&D^{\mathrm{Lk}(T)}(\kappa _1\otimes \cdots \otimes \kappa _{n})\\
=&\tilde D^{\mathrm{Lk}(T)} (K^{-m_{1,1}}\kappa _1\otimes \cdots \otimes K^{-m_{n,n}}\kappa _{n}).
\end{align*}
Since we have $K^{-m_{i,i}} \kappa _i\in \bar {U}_{q}^{ev0}$ by Lemma \ref{Cap}, the right hand side is contained in $\tilde D^{\mathrm{Lk}(T)}(\bar U_{q}^{ev0})^{\otimes n}$. This completes the proof.
\end{proof}

\begin{lemma}\label{dia}
For every  $i=1,\ldots,n,$
we have
\begin{align*}
J(P_i,s )^{\diamondsuit }\sim 1.
\end{align*}
\end{lemma}
If we assume Lemma \ref{dia}, then Proposition \ref{WW} follows from
\begin{align*}
J(P,s)\sim J(P,s )^{\bullet }J(P,s )^{\diamondsuit }J(P,s)^{\circ }
\in \tilde  D^{\mathrm{Lk}(T)}(\bar U_q^{ev0})^{\otimes n}\cdot (U_{\mathbb{Z},q}^{ev})^{\otimes n}\subset
  \tilde  D^{\mathrm{Lk}(T)} (U_{\mathbb{Z},q}^{ev})^{\otimes n},
\end{align*}
 by  (\ref{noten}) and  Lemma \ref{bu}.
 
\begin{proof}[Proof of Lemma \ref{dia}.]
For a crossing $c$ of $(P,s)$,  we denote by $E_c$ (resp. $F_c$) the white dot 
on the  over (resp. under) strand labeled by $e^{s(c)}$ (resp. $\tilde F^{(s(c))}$). 
We slide  those  white dots  to the start points of strands of $P$, and count the  powers of $K$ labeled to the  white diamonds
on each strands.

Note that  each time we exchange  $E_c$ with one of the two dots connected by dashed line,  labeled by $D^{\pm 1}$,  a white diamond  labeled by $K^{\mp s(c)}$  appears  next to the other dot, see Figure \ref{fig:D2} again.
Similarly,  if we exchange   $F_c$ with one of the two  dots labeled by $D^{\pm 1}$, then  a white diamond  labeled by $K^{\pm s(c)}$  appears 
next to  the other dot. 

Let $p_i(E _c)$ denotes the number of times $E_c$ traverses the strand $P_i$
during the sliding process.
Define $p_i(F_c)$ similarly.
Then we have  $J(P_i,s )^{\diamondsuit }=K^{d_i}$, where
$$
d_i\equiv \sum_{c\in C(P)}s(c)(p_i(E _c)+p_i(F_c)) \ \  \pmod 2.
$$
Hence it is enough to prove that  $p_i(E _c)+p_i(F_c)$ is even for  each crossing $c$.
We prove the assertion with  three types of  crossings as follows.
\begin{itemize}
\item[(i)] Self crossings of $P_i$. 
  \item[(ii)] Crossings of $P_j$ with $P_l$ for  $j\neq i, l\neq i$.
\item[ (iii)] Crossings  of $P_i$ with  $P_j$ for  $j\neq i$.
\end{itemize}

Color black or white, in chessboard fashion, the regions of the complements of $P_i$ in the rectangle so that
the outermost region is colored  white. For example, see Figure \ref{fig:LLink3}.
\begin{figure}
\centering
\includegraphics[width=8cm,clip]{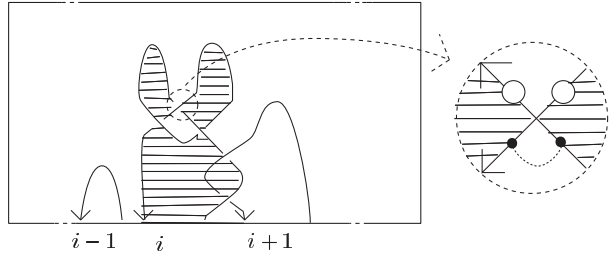}
\caption{A  diagram $P=P_1\cup \cdots \cup P_n$ colored by chessboard fashion  associated to $P_i$.
We depict only the $(i-1)$, $i$, and $(i+1)$th component. }\label{fig:LLink3}
\end{figure}%
Divide the strand $P_i$ into two parts $B_i$ and $W_i$, each consisting of segments bounded by 
self crossing points or the boundary points of $P_i,$ such that if one goes along
a segment in $W_i$ (resp. $B_i$) to the start point of $P_i,$ then one sees a white (resp. black) region on the left.

Note that the boundary  points of the strand   $P_l$,  $i\neq l$, are contained in the white region, and
those  of $P_i$ are contained  in $W_i$.
\begin{itemize}
\item[(i)]For a  self crossing $c$ of $P_i$.

Note that when we trace along $P_i$ from the end point to the start point, every time we traverse the self crossing of $P_i$, $B_P$ and $W_P$  appear one after the other.
For every self crossing $c\in P_i$, both $E _c$ and $F _c$ are  either in $B_P$ or in $W_P$.
Hence if we slide  $E _c$ and $F _{c}$ to the start point, then the parities of $p_i(E _c)$ and  $p_i(F _c)$ are the same.
Thus, $p_i(E _c)+p_i(F_c)$ is even.
\item[(ii)]For a  crossing $c$ of $P_j$ and $P_l$ with $j\neq i, l\neq i$.

If the crossing $c$ is in the white region, then both $p_i(E _c)$ and $p_i(F _{c})$ are  even.
If $c$ is in the black region, then  both $p_i(E _c)$ and $p_i(F _c)$ are  odd.
Hence $p_i(E _c)+p_i(F_c)$ is even in both cases.

\item[(iii)]For a crossing $c$ of $P_i$ and $P_j$ with $j\neq i$.

See Figure \ref{fig:cases}.
There are  four types of crossings such that  whether the white dot on $P_i$ is in $W_i$ or in $B_i$, and 
whether the white dot on $P_j$ is in the  white region or in the black region.
\begin{figure}
\centering
\includegraphics[width=10cm,clip]{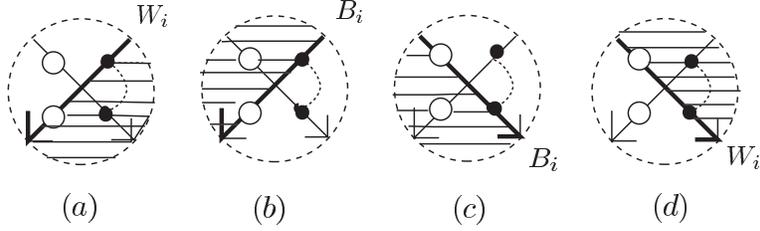}
\caption{The four types of crossings.}\label{fig:cases}

\end{figure}%
We assume $P_i$ is the over strand, i.e., $E_c$ is attached on $P_i$.
The other case  is almost the same.
For $(a)$,  since $E_c $ starts and ends in $W_i$, $p_i(E _c)$ is even.
Similarly,   since $F_c$ starts and ends in the white region, $p_i(F_c)$ is even.
Thus, $p_i(E _c)+p_i(F_c)$ is even.
 For the other three cases, in a similar way, we can prove that the parities of $p_i(E _c)$ and $p_i(F _c)$ are the same. 
 Hence $p_i(E_c)+p_i(F _c)$ is even.
\end{itemize}
Therefore we have $J(P_i,s)^{\diamondsuit } \sim 1$ for $i=1,\ldots, n$, this completes the proof.
\end{proof}
\begin{remark}
As defined in \cite{H2}, let  $\mathcal{U}_q^{ev}\subset U_{\mathbb{Z},q}$ denote the subalgebra  of $U_h$ freely generated over $\mathbb{Z}[q,q^{-1}]$  by the 
elements $\tilde F^{(i)}K^{2j}e^k$ for $i,k\geq 0, j\in \mathbb{Z}$.
Note that the right hand sides of (\ref{rem1}) and (\ref{rem2}) are in $(\mathcal{U}_q^{ev})^{\otimes 2}$.
This  implies  a result stronger  than Proposition \ref{WW} ;
\begin{align*}
J(P,s)\in \tilde D^{\mathrm{Lk}(T)}(\mathcal{U}_q^{ev})^{\otimes n}.
\end{align*}
This implies the following, which is proved by Habiro when $\mathrm{Lk}(T)=0$ in the other way.
\begin{align*}
J_T\in \tilde D^{\mathrm{Lk}(T)}(\tilde{ \mathcal{U}}_q^{ev})^{\tilde \otimes n},
\end{align*}
where $(\tilde {\mathcal{U}_q}^{ev})^{\tilde \otimes n}$ is the Habiro's  completion of $(\mathcal{U}_q^{ev})^{\otimes n}$ in \cite{H2}.
\end{remark}

\subsection{The universal $sl_2$ invariant of ribbon bottom tangles.}\label{ribbon}
Habiro \cite{H2} studied  the universal   $sl_2$ invariant of  $1$-component ribbon bottom tangles.
We  generalize those  to  $n$-component ribbon bottom tangles for $n\geq 1$. 

For $T\in BT_{i+j+2}$, $i,j\geq 0$,  let $(\ad_b)_{i,j}(T)\in BT_{i+j+1}$ and $(\mu_b)_{(i,j)}(T)\in BT_{i+j+1}$  denote the bottom tangles as depicted in Figure \ref{fig:ad2}. We use  the following lemma.

\begin{figure}
\centering
\includegraphics[width=9cm,clip]{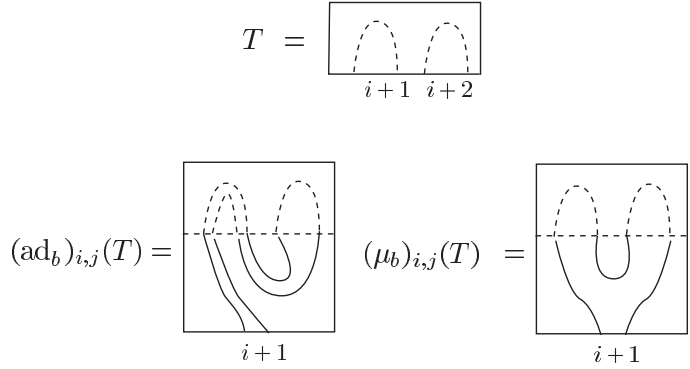}
\caption{A bottom tangle $T\in BT_{i+j+2}$ and the bottom tangles  $(\ad_b)_{(i,j)}(T)$, $(\mu_b)_{(i,j)}(T)\in BT_{i+j+1}$.
We  depict only the $(i+1), (i+2)$th components of $T$, and the $(i+1)$th components of  $(\ad_b)_{(i,j)}(T)$, $(\mu_b)_{(i,j)}(T)$.}\label{fig:ad2}
\end{figure}
\begin{lemma}[Habiro \cite{H1}]\label{adj}
For every bottom tangle $T\in BT_{i+j+2}$,  $i,j\geq 0$, we have
\begin{align*}
&J_{(\ad_b)_{i,j}(T)}=\ad_{i,j}(J_T),
\\
&J_{ (\mu _b)_{i,j}(T)}=\mu_{i,j}(J_T),
\end{align*}
where we set
\begin{align*}
\ad_{i,j}=\id^{\otimes i}\otimes \ad\otimes \id^{\otimes j}\colon\ U_h^{\hat \otimes i+j+2}\rightarrow U_h^{\hat \otimes i+j+1},
\\
\mu_{i,j}=\id^{\otimes i}\otimes \mu\otimes \id^{\otimes j}\colon\ U_h^{\hat \otimes i+j+2}\rightarrow U_h^{\hat \otimes i+j+1}.
\end{align*}
Here $\mu\colon\ U_h\hat \otimes U_h\rightarrow U_h $ is the multiplication of $U_h$.
\end{lemma}
For a $2k$-component bottom tangle $W=W_1\cup \cdots \cup W_{2k}\in BT_{2k}, k\geq 0$, set 
\begin{align*}
W^{ev}=\bigcup _{i=1}^kW_{2i} \in BT_{k},\ \ \text{and} \ \ 
W^{odd}=\bigcup _{i=1}^kW_{2i-1}\in BT_{k}.
\end{align*}
For a diagram $P$ of $W$,  let $P^{ev}$ (resp. $P^{odd}$) denote the part of the diagram $P$ corresponding to $W^{ev}$
(resp. $W^{odd}$).
We say  a bottom tangle $W\in BT_{2k}$ is \textit{even-trivial} if $W^{ev}$ is a trivial bottom tangle. 
For example, see Figure \ref{fig:add1}.
We also say  a diagram $P$ of $W$ is \textit{even-trivial}  if and only if $P^{ev}$ has no self crossings.
Note that a bottom tangle $W$ has an even-trivial diagram  if and only if $W$ is even-trivial.
\begin{figure}

\centering
\includegraphics[width=8cm,clip]{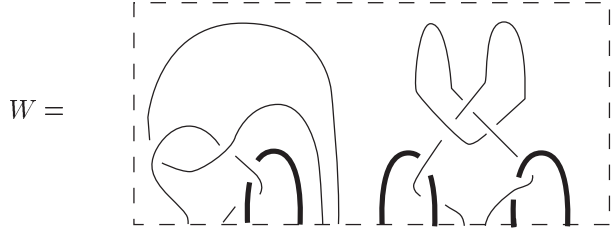}
\caption{An even-trivial bottom tangle $W\in BT_6$.
Here $W^{ev}$ is depicted with thick lines.}\label{fig:add1}
\end{figure}

The following lemma is almost the same as \cite[Theorem 11.5]{H1}.
\begin{proposition}\label{r1}
For any bottom tangle $T\in BT_n$, the following conditions are equivalent.
\begin{itemize}
\item[(1)]
$T$ is a ribbon bottom tangle.
\item[(2)]
There is  an even-trivial bottom tangle  $W\in BT_{2k}, k\geq 0$, and there are integers $ N_1,\ldots , N_n\geq 0$ satisfying $N_1+\cdots+N_n=k$,  such that
\begin{align}
T=\mu _b^{[N_1,\ldots,N_n]}\ad_b^{\otimes k}(W),\label{data}
\end{align}
\end{itemize}
where 
$$
\ad_b^{\otimes k}\colon\;BT_{2k}\rightarrow BT_{k}
$$
is as depicted in Figure \ref{fig:adjoint}, and 
$$
 \mu _b^{[N_1,\ldots,N_n]}\colon\;BT_{N_1+\cdots +N_n}\rightarrow BT_n
$$
is as depicted in Figure \ref{fig:mu}.
\end{proposition}
If (\ref{data}) holds, then we call $(W; N_1,\ldots , N_n)$ a \textit{ribbon data} for $T$. 
For example, the ribbon bottom tangle $\mu ^{[ 1,2,0]}(ad_b)^{\otimes 3}(W)\in BT_3$ with the ribbon data 
$(W\in BT_3; 1,2,0)$, where $W$ is the bottom tangle  in Figure \ref{fig:add1}, is as  depicted in Figure \ref{fig:add2}.

\begin{proof}[Proof of Proposition \ref{r1}]
In view of Proposition \ref{rs}, the proof is almost the same as  that of  Theorem 11.5 in \cite{H1}.
\end{proof}
\begin{figure}

\centering
\includegraphics[width=10cm,clip]{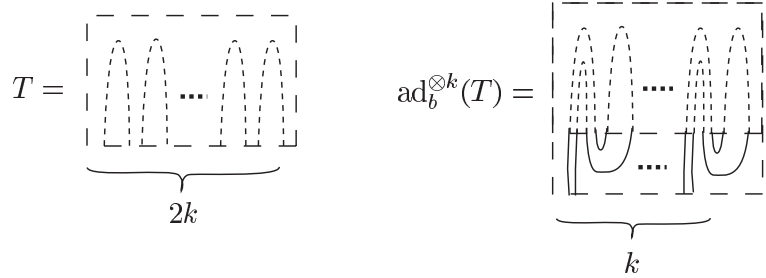}
\caption{A bottom tangle $T\in BT_{2k}$ and the bottom tangle $\ad_b^{\otimes k}(T)\in BT_{k}$.}\label{fig:adjoint}
\end{figure}\begin{figure}

\centering
\includegraphics[width=11cm,clip]{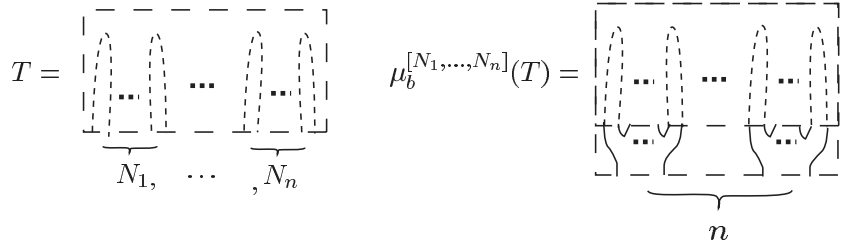}
\caption{A bottom tangle $T\in BT_{k}$ and the bottom tangle $\mu_b^{[N_1,\ldots,N_n]}(T)\in BT_{n}$. }\label{fig:mu}
\end{figure}
\begin{figure}

\centering
\includegraphics[width=10cm,clip]{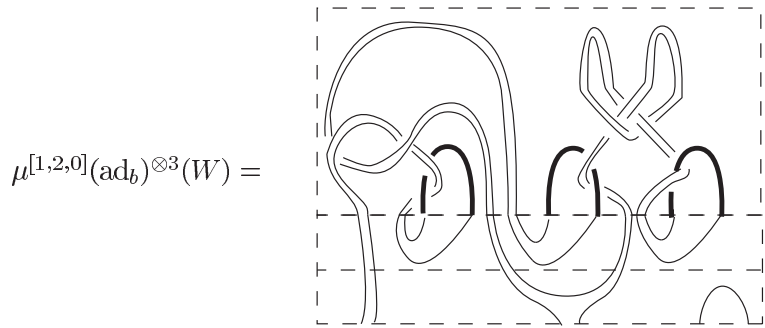}
\caption{The ribbon bottom  tangle $\mu ^{[ 1,2,0]}(ad_b)^{\otimes 3}(W)\in BT_3$ for the even-trivial bottom tangle $W\in BT_3$ in Figure \ref{fig:add1}.}\label{fig:add2}
\end{figure}

For $n\geq 1$, let  $$\mu  ^{[ n]}\colon\;U_h^{\hat \otimes n}\rightarrow U_h, \ \ x_1\otimes \cdots \otimes x_n\mapsto x_1x_2\cdots x_n$$ 
denote the $n$-input multiplication. 
For integers $ N_1,\ldots , N_n\geq 0$, $ N_1+\cdots+N_n=k$, set
\begin{align*}
\mu ^{[N_1,\ldots,N_n]}=\mu ^{[N_1]}\otimes \cdots \otimes \mu ^{[N_n]} \colon\; U_h^{\hat \otimes k}\rightarrow U_h^{\hat \otimes n}.
\end{align*}
\begin{proposition}\label{rib}
Let $T\in BT_n$ be a ribbon bottom tangle and  $(W\in BT_{2k}; N_1,\ldots,N_n) $  a ribbon data for $T$.
 Then we have
 \begin{align*}
 J_T=\mu ^{[N_1,\ldots,N_n]}\ad^{\otimes k}(J_W).
 \end{align*}
\end{proposition}
\begin{proof}
By  Lemma \ref{adj}, we have 
\begin{align*}
J_{\ad_b^{\otimes k}(T)}&=\ad^{\otimes k}(J_T),
\end{align*}
for $T\in BT_{2k}$, and 
\begin{align*}
J_{\mu_b ^{[N_1,\ldots,N_n]}(T)}&=\mu ^{[N_1,\ldots,N_n]}(J_T),
\end{align*}
for $T\in BT_k$. This implies the assertion.
\end{proof}
\section{Proof of Theorem \ref{1}.}\label{proof}
In this  section, we prove  Theorems \ref{1}.
Let $T\in BT_n$, $n\geq 0$, be a ribbon bottom tangle, and $(W\in BT_{2k}; N_1,\ldots ,N_n)$, $k\geq 0$,  a ribbon data for $T$.
Let $P_W$ be an even-trivial diagram of  $W$,  and $s\in \mathcal{S}(P_W)$  a state.
We use this setting  throughout this section.
The  proof of Theorem \ref{1} is outlined as follows.

First,  we prove the following proposition.
\begin{proposition} \label{W}
We have
\begin{align*}
J(P_W,s) \in \tilde D^{\mathrm{Lk}(T)}(U_{\mathbb{Z}, q}^{ev}\otimes \bar {U}_q^{ev})^{\otimes k}.
\end{align*}
\end{proposition}
Then we consider the contribution of  $\tilde D^{\mathrm{Lk}(T)}$ to   the adjoint action, and we construct
an element $\tilde J(P_W,s)\in (U_{\mathbb{Z}, q}^{ev}\otimes \bar {U}_q^{ev})^{\otimes k}$ such that
\begin{align}
\ad^{\otimes k}(J(P_W,s))=\ad^{\otimes k}(\tilde J(P_W,s)).\label{J}
\end{align}
Thus, by Proposition \ref{Habi}, we have 
\begin{align}
\ad^{\otimes k}(J(P_W,s))\in (\bar {U}_q^{ev})^{\otimes k}.
\end{align}
Finally, we  define a  completion $(\bar U_q^{ev})\;\hat  {}^{\;\hat  \otimes k}$ of $(\bar U_q^{ev})^{  \otimes k}$ and  prove Theorem \ref{1}, i.e., we prove
\begin{align*}
J_T=\mu ^{[N_1,\ldots ,N_n]}\sum_{s\in \mathcal{S}(P_W)}\ad^{\otimes k}(J(P_W,s))\in (\bar U_q^{ev})\;\hat  {}^{\;\hat  \otimes n}.
\end{align*}
\subsection{Proof of  Proposition \ref{W}.}
We modify the proof of Proposition \ref{WW}.
The key to the proof is the fact 
\begin{align*}
K^{m}R^+_{m[1]}\otimes  R^+_{m[2]} ,  R^-_{m[1]}\otimes K^mR^-_{m[2]}\in (U_{\mathbb{Z},q}^{ev}\otimes \bar U_q^{ev})\cap (\bar U_q^{ev} \otimes U_{\mathbb{Z},q}^{ev}),
\end{align*}
which follows  from (\ref{rem1}) and (\ref{rem2}).
Since $P_W$ is even-trivial, 
the set $C(P_W)$  of the crossings of $P_W$ is the disjoint union of two  subsets 
\begin{center}
$C^{eo}=\{$ crossings of $P_W^{ev}$ with $P_W^{odd}\}$,
\end{center}
 and
\begin{center} $C^{oo}=\{$ crossings of $P_W^{odd}$ with $P_W^{odd}\}$.    \end{center}
Thus, on the decorated diagram $(P_W,s)$, we can assume that the element attached to the  white    dot on   $P_W^{ev}$ (resp.  $P_W^{odd}$) is contained in $\bar U_q^{ev}$ (resp. $U_{\mathbb{Z},q}^{ev}$).
 For example, we attach  elements  to positive crossings as depicted in Figure \ref{fig:Pev}. 
 Then for the decorated diagram $(P_W,s)^{\circ }$, we have 
 \begin{align}
J(P_W,s)^{\circ }\in (U_{\mathbb{Z}, q}^{ev}\otimes \bar {U}_q^{ev})^{\otimes k}.\label{note}
 \end{align}
 The rest is analogous to  the  proof of  Proposition \ref{WW}. 
 \begin{figure}
\centering
\includegraphics[width=13cm,clip]{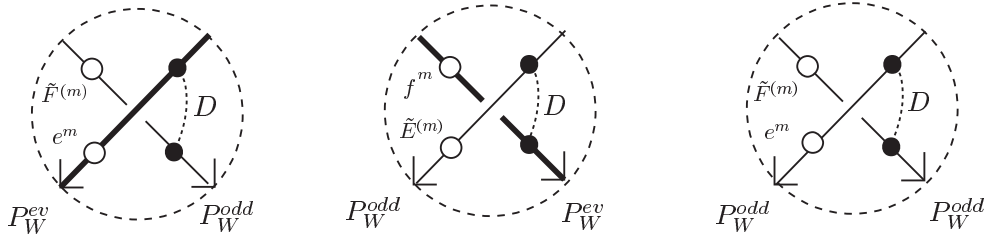}
\caption{The three types of positive crossings. We work up to multiplication by  $\pm q^{j}, K^{2j} (j\in \mathbb{Z})$.}\label{fig:Pev}
\end{figure}%
\subsection{The element $\tilde J(P_W,s)$.}
In this subsection, we construct the element $\tilde J(P_W,s)\in (U_{\mathbb{Z}, q}^{ev}\otimes \bar {U}_q^{ev})^{\otimes k}$ satisfying (\ref{J}).
\begin{lemma} \label{lem1}
For homogeneous elements $x, y\in U_h$, we have 
\begin{itemize}
\item[(i)]$\sum (D^{\pm}_{[1]} \triangleright x)\otimes D^{\pm}_{[2]}= x \otimes K^{\pm |x|}, $
\item[(ii)]$\sum (D^{\pm}_{[1]}\triangleright x)\otimes (D^{\pm}_{[2]}\triangleright y)= q^{ \pm|x||y|}x \otimes y, $ and
\item[(iii)]$(v^{ H^2/2}K)^{\pm 1}\triangleright x= q^{\pm |x|(|x|+1)}x$.
\end{itemize}
\end{lemma}
\begin{proof}
We prove the  formulas for the positive signs. Then the other cases  are similar.
By the formulas (\ref{d1})--(\ref{exD}), we have
\begin{itemize}
\item[(i)] $\sum (D^{+}_{[1]} \triangleright x)\otimes D^{+}_{[2]}=\sum (D^{+}_{[1]}xD^{-}_{[1]}) \otimes (D^{+}_{[2]}D^-_{[2]})=x \otimes K^{ |x|}$.
\end{itemize}
Using (i), we obtain
\begin{itemize}
\item[(ii)] $\sum (D^{+}_{[1]}\triangleright x)\otimes (D^{+}_{[2]}\triangleright y)= \sum x\otimes (K^{ |x|}\triangleright y)=q^{|x||y|}x \otimes y$, and 
\item[(iii)]$(v^{ H^2/2}K)\triangleright x=\sum (D^{+}_{[1]}D^{+}_{[2]}K)\triangleright x
= q^{|x|}(D^{+}_{[1]}D^{+}_{[2]}\triangleright x)=q^{|x|}(K^{|x|}\triangleright x)=q^{|x|(|x|+1)}x.$
\end{itemize}
\end{proof}
\begin{lemma}\label{DD}
For $k\geq 0$, let $M=(m_{i,j})_{1\leq i,j\leq 2k}$ be a symmetric integer matrix of size $2k$, satisfying $m_{2i,2j}=0$ for $1\leq i,j\leq k$. 
Let  $X=x_1\otimes \cdots\otimes x_{2k}\in U_h^{\otimes 2k}$ be the tensor product of   homogeneous elements $x_1, \ldots,x_{2k}\in U_h$.
We have
\begin{align*}
&\ad^{\otimes k}(\tilde D^{M} X)=
q^{N(M,X)}\ad^{\otimes k}\big(( 1\otimes K^{2a_1(M,X)}\otimes \cdots \otimes 1\otimes K^{2a_m(M,X)})X\big),
\intertext{where if we set  $X_i=x_{2i-1}\triangleright x_{2i}$, then}
a_{i}(M,X)&=\sum_{1\leq j\leq k}m_{2i, 2j-1}|X_j|,
\\
N(M,X)&=\sum_{1\leq i< j\leq k} 2m_{2i-1,2j-1} |X_i||X_j|+\sum_{1\leq i\leq k}m_{2i-1,2i-1}|X_i|(|X_i|+1).
\end{align*}
\end{lemma}
Here $|X_i|=|x_{2i-1}|+|x_{2i}|$ is the degree of $X_i$ defined in Section \ref{pre}.
\begin{proof}We use  induction on $\sum_{1\leq i,j\leq 2k}|m_{ij}|$.
If $\sum_{1\leq i,j\leq 2k}|m_{ij}|=0$, i.e., $M=0$, then the claim is clear.
Let us assume $M\neq0$. 
Then there is a matrix  $M'$  satisfying the assertion, and   either 
\begin{align*}
M&=M'\pm (1_{2i,2j-1}+1_{2j-1,2i}), \quad \text{for}  \ 1\leq i \neq j\leq k, \ \text{or}
\\
M&=M'\pm (1_{2i-1,2j-1}+1_{2j-1,2i-1}), \quad \text{for}  \ 1\leq i \neq j\leq k, \ \text{or}
\\
M&=M'\pm 1_{2i-1,2i-1}, \quad \text{for}  \ 1\leq i\leq k,
\end{align*}
where $1_{i,j}$ is the matrix of size $2k$ such that the $(i,j)$-component is $1$ and the others are  $0$. 
Note that 
\begin{align*}
&\tilde D^{M'\pm (1_{i,j}+1_{j,i})}=\tilde D^{M'}D_{i,j}^{\pm 2}, \quad  \text{for} \ \ 1\leq i\neq j\leq 2k, \ \text{and} 
\\
&\tilde D^{M' \pm 1_{ii}}=\tilde D^{M'}(v^{ H^2/2}K)^{\pm 1}_i, \quad \text{for} \ \ 1\leq i\leq 2k.
\end{align*}
Then the following  formulas using Lemma  \ref{lem1} imply the assertion.
\begin{align*}
\begin{split} 
\ad^{\otimes k}(D^{\pm 1}_{2i,2j-1}X)&=X_1\otimes \cdots \otimes (x_{2i-1}\triangleright D^{\pm}_{[1]}x_{2i})\otimes \cdots \otimes (D^{\pm}_{[2]}\triangleright X_j)\otimes \cdots \otimes X_k
\\
&=X_1\otimes \cdots \otimes (x_{2i-1}\triangleright K^{\pm |X_j|}x_{2i})\otimes \cdots \otimes  X_j\otimes \cdots \otimes X_k,
\end{split}
\\
\ad^{\otimes k}(D^{\pm 1}_{2i,2i-1}X)&=X_1\otimes \cdots \otimes (D^{\pm}_{[2]}x_{2i-1}\triangleright D^{\pm}_{[1]}x_{2i})\otimes \cdots \otimes \cdots \otimes X_k
\\
&=X_1\otimes \cdots \otimes (x_{2i-1}\triangleright K^{\pm |X_i|}x_{2i})\otimes \cdots \otimes  X_j\otimes \cdots \otimes X_k,
\\
\begin{split}
\ad^{\otimes k}(D^{\pm 1}_{2i-1,2j-1}X)&=X_1\otimes \cdots \otimes (D^{\pm}_{[1]}\triangleright X_i)\otimes \cdots \otimes (D^{\pm}_{[2]}\triangleright X_j)\otimes \cdots\otimes  X_k
\\
&=q^{\pm |X_i||X_j|}X_1\otimes \cdots \otimes X_i \otimes \cdots \otimes  X_j \otimes \cdots \otimes X_k,
\end{split}
\\
\begin{split} 
\ad^{\otimes k}\big((v^{ H^2/2}K)^{\pm 1}_{2i-1}X\big)&=
X_1\otimes \cdots \otimes ((v^{ H^2/2}K)^{\pm 1}\triangleright X_i)\otimes  \cdots \otimes  X_k
\\
&=q^{\pm |X_i|(|X_i|+1)}X_1\otimes \cdots \otimes X_i \otimes \cdots \otimes X_k,
\end{split}
\end{align*}
for $1\leq i\neq j\leq k$.
\end{proof}
By Proposition \ref{W}, we have
$$
X:=(\tilde D^{\mathrm{Lk}(W)})^{-1}J(P_W,s)\in (U_{\mathbb{Z}, q}^{ev}\otimes \bar {U}_q^{ev})^{\otimes k}.
$$ 
Since the linking matrix  $\mathrm{Lk}(W)$ of $W$ satisfies the assumption of Lemma \ref{DD},  we obtain the element $\tilde J(P_W,s)\in (U_{\mathbb{Z}, q}^{ev}\otimes \bar {U}_q^{ev})^{\otimes k}$  satisfying (\ref{J}), such that
\begin{align*}
\tilde J(P_W,s):=q^{N}( 1\otimes K^{2a_1}\otimes \cdots \otimes 1\otimes K^{2a_k})X,
\end{align*}
where we set 
$$N=N(\mathrm{Lk}(W),X),
$$
and 
$$
a_i=a_i(\mathrm{Lk}(W),X),$$
for $i=1,\ldots, k$, as in Lemma \ref{DD}.
\subsection{Filtrations of $\bar U_{q}^{ev}$.}
In this subsection, we define two  filtrations $\{A_p\}_{p\geq 0}$  and $\{C_p\}_{p\geq 0}$ of $\bar U_{q}^{ev}$, which are cofinal with each other.
We give four equivalent definitions for $\{A_p\}_{p\geq 0}$, and  two  for $\{C_p\}_{p\geq 0}$.

For a subset $X\subset \bar  U_q^{ev}$, let $\langle X \rangle_{ideal} $ denote the two-sided ideal of $\bar  U_q^{ev}$ generated by $X$.
For $p\geq 0$, set 
\begin{align*}
&A_p=\langle U_{\mathbb{Z},q}\triangleright e^p  \rangle_{ideal},
\\
&A'_p=\langle U_{\mathbb{Z},q}\triangleright f^p \rangle_{ideal},
\\
&B_p=\langle K^p(U_{\mathbb{Z},q}\triangleright K^{-p}e^p) \rangle_{ideal},
\\
&B'_p=\langle K^p(U_{\mathbb{Z},q}\triangleright f^pK^{-p})  \rangle_{ideal},
\\
&C_p=\langle 
\sum_{p'\geq p}(U_{\mathbb{Z},q}\tilde E^{(p')}\triangleright \bar U_q^{ev}\big)
 \rangle_{ideal},
\\
&C'_p=\langle 
\sum_{p'\geq p}(U_{\mathbb{Z},q}\tilde F^{(p')}\triangleright \bar U_q^{ev}\big)
 \rangle_{ideal}.
\end{align*}
\begin{proposition}\label{EQ}
For $p\geq 0$,  we have
\begin{align*}
A_p=A'_p=B_p=B'_p.
\end{align*}
\end{proposition}
\begin{proof}
By the formulas
\begin{align}\label{hi}
f^pK^{-p}=(-1)^pq^{-p^2}\tilde F^{(2p)}\triangleright K^{-p}e^p \in U_{\mathbb{Z},q} \triangleright K^{-p}e^p, \\
K^{-p}e^{p}= (-1)^pq^{p^2} \tilde E^{(2p)}\triangleright f^pK^{-p} \in U_{\mathbb{Z},q} \triangleright f^pK^{-p},
\end{align}
we have  $B_p=B'_p$.
We prove $A_p=B_p$, then $A'_p=B'_p$ is similar.
By  Proposition \ref{Habi}, we have
\begin{align*}
K^p(U_{\mathbb{Z},q}\triangleright K^{-p}e^p)
&\subset  K^p(U_{\mathbb{Z},q}\triangleright K^{-p})\cdot (U_{\mathbb{Z},q}\triangleright e^p)
\\
&\subset  \bar U_q^{ev}(U_{\mathbb{Z},q}\triangleright e^p)\subset A_p.
\end{align*}
Hence we have  $B_p\subset A_p$. Conversely, we have
\begin{align*}
U_{\mathbb{Z},q}\triangleright e^p= U_{\mathbb{Z},q}\triangleright K^pK^{-p}e^p
&\subset  (U_{\mathbb{Z},q}\triangleright K^p)\cdot (U_{\mathbb{Z},q}\triangleright K^{-p}e^p)
\\
&\subset  \bar U_q^{ev}K^p(U_{\mathbb{Z},q}\triangleright K^{-p}e^p)\subset B_p.
\end{align*}
Hence we have  $A_p\subset B_p$, this completes the proof.
\end{proof}
\begin{proposition}\label{EQQ}
\begin{itemize}
\item[(i)]
For $p\geq 0$, we have $
C_p=C'_p.$
\item[(ii)]
For $p\geq 0$, we have $C_{2p}\subset A_p$.
\item[(iii)]
If $p\geq 0$ is even, then we have $C_{2p}=A_p$.
\end{itemize}
\end{proposition}
\begin{proof}

(i) We prove    $C_p\subset C'_p$, then  $C_p\supset C'_p$ is similar.
Using the formula
\begin{align*}
\tilde E^{(2p)}\triangleright \tilde F^{(p)}K^{-p}=(-1)^pq^{-\frac{1}{2}p(p+1)}K^{-p}\tilde E^{(p)},
\end{align*}
 we have
\begin{align*}
U_{\mathbb{Z},q}\tilde E^{(p)}&\subset U_{\mathbb{Z},q}\big(\tilde E^{(2p)}\triangleright \tilde F^{(p)}K^{-p}\big)
\\
&\subset U_{\mathbb{Z},q}\tilde F^{(p)}U_{\mathbb{Z},q}.
\end{align*}
Hence we have
\begin{align*}
U_{\mathbb{Z},q}\tilde E^{(p)}\triangleright \bar U_q^{ev}&
\subset U_{\mathbb{Z},q}\tilde F^{(p)}U_{\mathbb{Z},q}\triangleright  \bar U_q^{ev}
\\
&\subset U_{\mathbb{Z},q}\tilde F^{(p)}\triangleright \bar U_q^{ev}.
\end{align*}
This completes the proof.

(ii)
In view of Lemma \ref{F2}, it is enough to prove that
\begin{align*}
\tilde E^{(p')}\triangleright f^{i_1}K^{2i_2}e^{i_3}\subset A_p,
\end{align*}
for $p'\geq  2p.$
If $i_1\geq p'\geq p$, then the assertion follows from
\begin{align*}
U_{\mathbb{Z},q} \triangleright f^{i_1}K^{2i_2}e^{i_3}\subset (U_{\mathbb{Z},q}\triangleright f^{i_1})\bar U_{q}^{ev}\subset A_p'=A_p.
\end{align*}
If $i_1<p'$, then we have
\begin{align*}
\tilde E^{(p')}\triangleright f^{i_1}K^{2i_2}e^{i_3}&\in \langle U_{\mathbb{Z},q}\triangleright  f^{i_1} \rangle_{ideal} \cap  \langle e^{i_3+p'-i_1}\rangle_{ideal},
\\
&\subset A'_{i_1}\cap A_{i_3+p'-i_1}
\\
&\subset A_{\max\{i_1,i_3+p'-i_1\}},
\end{align*}
where the $\in $  follows from the formula (\ref{KE}), and the last $\subset $ follows from  Proposition  \ref{EQ}.
Hence  the assertion follows from 
\begin{align*}
\max\{i_1, i_3+p'-i_1\}\geq \frac{i_3+p'}{2}\geq p.
\end{align*}

(iii)
If $p\geq 0$ is even, then  we have 
\begin{align*}
K^p(U_{\mathbb{Z},q}\triangleright K^{-p}e^p)&=(-1)^pq^{p^2}K^p(U_{\mathbb{Z},q}\triangleright (\tilde E^{(2p)}\triangleright f^pK^{-p}))
\\
&\subset \langle U_{\mathbb{Z},q}\tilde E^{(2p)}\triangleright \bar U_q^{ev}\rangle_{ideal}\subset C_{2p},
\end{align*}
from (\ref{hi}).
Hence we have $C_{2p}\supset B_p(=A_p)$, this completes the proof.
\end{proof}
\begin{corollary}
For $p\geq 0$, we have
\begin{align*}
C_{2p}\subset h^pU_h.
\end{align*}
\end{corollary}
\begin{proof}
Since $e^p\subset h^pU_h$, we have $C_{2p}\subset A_p\subset h^pU_h$ by Proposition \ref{EQQ}.
\end{proof}
\subsection{The  completion $(\bar U_q^{ev})\;\hat  {}^{\;\hat  \otimes n}$ of $(\bar U_q^{ev})^{ \otimes n}$.}
In this subsection we define the completion $(\bar U_q^{ev})\;\hat  {}^{\;\hat  \otimes n}$ of $(\bar U_q^{ev})^{ \otimes n}$, and prove Theorem \ref{1}.
Let  $(\bar U_q^{ev})\;\hat{} $ denote the completion in $U_h$ of $\bar U_q^{ev}$ with respect to the decreasing filtration 
$\{C_p\}_{p\geq 0}$, i.e.,  $(\bar U_q^{ev})\;\hat{}$ is the image of the homomorphism
\begin{align*}
\varprojlim _{p}  \bar U_q^{ev}/C_p\rightarrow U_h.
\end{align*}
induced by  the inclusion $\bar U_q^{ev} \subset U_h$, which is well defined  since $C_{2p}\subset h^pU_h$ for $p\geq 0$.
For $n\geq 1$, we define a filtration $\{C_p^{(n)}\}_{p\geq 0}$ for $(\bar U_q^{ev})^{ \otimes n}$ by
\begin{align*}
C_p^{(n)}= \sum_{j=1}^n \bar U_q^{ev} \otimes \cdots \otimes \bar U_q^{ev}\otimes 
 C_p \otimes \bar U_q^{ev}\otimes \cdots \otimes \bar U_q^{ev},
\end{align*}
where $C_p$ is at the $j$th position.
Define the completion $(\bar U_q^{ev})\;\hat  {}^{\;\hat  \otimes n}$ of $ (\bar U_q^{ev})^{ \otimes n}$ as the image of the homomorphism
\begin{align*}
\varprojlim_{p}\big((\bar U_q^{ev})^{ \otimes n}/C_p^{(n)}\big) \rightarrow  
U_h^{\hat \otimes n}.
\end{align*}
For $n=0$, it is natural to set 
\begin{align*}
C_p^{(0)}=\begin{cases}
\mathbb{Z}[q,q^{-1}] \quad \text{if} \ \ p=0,
\\
0 \quad  \quad \quad \quad  \text{otherwise}.
\end{cases}
\end{align*}
Thus, we have 
\begin{align*}
(\bar U_q^{ev})\;\hat  {}^{\;\hat  \otimes 0}=\mathbb{Z}[q,q^{-1}].
\end{align*}

Recall the setting mentioned at the beginning of this section.
For $i=1,\ldots,2k,$ let $P_i$ denote the part of $P_W$ corresponding to the $i$th component of $W=W_1\cup \cdots \cup W_{2k}$, 
and  $C(P_i)$  the set of the crossings on the component $P_i$.
For $p\geq 0$, we denote by $\mathcal{I}_p$ the two-sided ideal of $U_{\mathbb{Z},q}$ generated by
$\tilde E^{(p)}, \tilde F^{(p)}\in U_{\mathbb{Z},q}$.
For $s\in \mathcal{S}(P_W)$, set $|s|_i=\max\{ s(c) \ | \ c\in C(P_i)\}$.
\begin{lemma}\label{Q2}
For each $s\in \mathcal{S}(P_W)$, there are  elements $w_{2i-1}\in U_{\mathbb{Z},q}^{ev} \cap \mathcal{I}_{|s|_{2i-1}}$ and $w_{2i}\in \bar U_q^{ev}\cap \mathcal{I}_{|s|_{2i}}$ for $i=1,\ldots k$, such that 
\begin{align*}
&\tilde J(P_W,s)= w_1\otimes \cdots  \otimes w_{2k}.
\end{align*}
\end{lemma}
\begin{proof}
Let $(P_i,s)^{\circ }$ denote the decorated diagram with $P_i$ and white dots of $(P_W,s)^{\circ }$ on $P_i$
(see p\pageref{noten} for the definition of  $(P_W,s)^{\circ }$). 
Recall that   one of the elements $\tilde E^{(s(c))}$, $\tilde F^{(s(c))}$, $e^{s(c)}$,  $f^{s(c)}$ is labeled on a white dot on a crossings $c$ of the decorated diagram $(P_W,s)^{\circ }$.
Since  each of those elements is   contained in $\mathcal{I}_{s(c)}$, we have 
\begin{align*}
&J(P_{i},s)^{\circ } \in \mathcal{I}_{|s|_i}.
\end{align*}
Note that
\begin{align*}
\tilde J(P_W,s)&\sim (\tilde D^{\mathrm{Lk}(W)})^{-1}J(P_W,s)
\\
&\sim J(P_1,s)^{\circ }\otimes \cdots \otimes J(P_{2k},s)^{\circ },
\end{align*}
where $\sim $ means equality up to multiplication by  $\pm q^j,K^{2j} (j\in \mathbb{Z})$ on any tensorands.
This and Proposition \ref{W} complete the proof.
\end{proof}
\begin{proof}[Proof of Theorem \ref{1}.]
 Let $|s|=\max\{s(c) \ | \ c\in C(P_W)\}$  denote the maximal integer of the image of $s$.
Since every crossing of  $P_W$ has  at least one strand  in $P_W^{odd}$,
we can assume $s(c)=|s|$ for a crossing  $c$ that has  a strand  of $P_{2j-1}$,  $1\leq j\leq k$.
Take  elements $w_{2i-1}\in U_{\mathbb{Z},q}^{ev} \cap \mathcal{I}_{|s|_{2i-1}} $ and  $w_{2i}\in \bar U_q^{ev}\cap \mathcal{I}_{|s|_{2i}} $, $i=1,\ldots, k$, as in  Lemma \ref{Q2}.
We have
\begin{align*}
w_{2j-1}\in \mathcal{I}_{|s|}.
\end{align*} 
Since $
\mathcal{I}_{|s|} \triangleright \bar U_q^{ev}\subset  C_{|s|}
$,
we have
\begin{align*}
w_{2j-1}\triangleright w_{2j}\in C_{|s|}.
\end{align*}
In view of   Proposition \ref{Habi}, we have 
\begin{align*}
\ad^{\otimes k}(\tilde J(P_W, s))=\ad^{\otimes k}(w_1\otimes \cdots\otimes w_{2k})\in  C_{|s|}^{(k)}.
\end{align*}
Thus by Proposition \ref{rib},  we have 
\begin{align*}
J_T&=\mu^{[N_1,\ldots, N_n]}\ad^{\otimes k}(J_W)
\\
&=\sum_{l\geq 0}\sum_{s\in \mathcal{S}(P_W), |s|=l}\mu^{[N_1,\ldots, N_n]}\ad^{\otimes k}(\tilde J(P_W, s))\in (\bar U_q^{ev})\;\hat  {}^{\;\hat  \otimes n}.
\end{align*}
This completes the proof.
\end{proof}

\begin{remark}
Recall from \cite{H2} the $\mathbb{Z}[q,q^{-1}]$-subalgebra $(\bar U_q^{ev})\;\tilde {}^{\;\tilde \otimes n}$ of $U_h^{\hat \otimes n}$.
We can prove the inclusion $(\bar U_q^{ev})\;\hat  {}^{\;\hat  \otimes n}\subset (\bar U_q^{ev})\;\tilde {}^{\;\tilde \otimes n}$ as follows.
We have only to prove
$
C_{2p}\subset  \mathcal{F}_p(\mathcal{U}_q^{ev}),$
for $p\geq 0,$ where  $\mathcal{F}_p(\mathcal{U}_q^{ev})$ denote the two-sided ideal of $\mathcal{U}_q^{ev}$ generated by $e^p$.
In view of Proposition \ref{EQQ}, we have only to prove $
A_{p}\subset  \mathcal{F}_p(\mathcal{U}_q^{ev}).$

Set
\begin{align*}
\begin{bmatrix} H+i \\p\end{bmatrix}_q=\{H+i\}_{q,p}/\{p\}_q!,
\end{align*}
for  $i\in \mathbb{Z}$, $p\geq 0$.
One can show that
\begin{align*}
U_{\mathbb{Z},q}^{ev}=\bigoplus _{i,j\geq 0}\tilde F^{(i)}U_{\mathbb{Z},q}^{0ev}\tilde E^{(j)},
\end{align*}
where $U_{\mathbb{Z},q}^{0ev}$ is the $\mathbb{Z}[q,q^{-1}]$-subalgebra of $U_{\mathbb{Z},q}^{ev}$ generated by
the elements
$K^2,K^{-2},$ and $ \begin{bmatrix} H+i \\p\end{bmatrix}_q$ for $i\in \mathbb{Z}$, $p\geq 0$ (This fact is a variant of a well known fact on 
Lusztig's integral form $U_{\mathbb{Z}}$ \cite{L}).
Thus it is enough to prove that 
\begin{align*}
\tilde F^{(i)}g\tilde E^{(j)}\triangleright e^p &\subset \mathcal{F}_p(\mathcal{U}_q^{ev}),
\end{align*}
for $i,j\geq 0$ and $g\in U_{\mathbb{Z},q}^{0ev}$.
For a homogeneous element $x\in U_h$, we have $U_{\mathbb{Z},q}^{0ev}\triangleright x\subset \mathbb{Z}[q,q^{-1}]x$ since
\begin{align*}
K\triangleright x=q^{|x|}x,
 \quad \begin{bmatrix}H+k
\\ l
\end{bmatrix}_q\triangleright x= \begin{bmatrix}2|x|+k
\\ l
\end{bmatrix}_qx,
\end{align*}
for $k\in \mathbb{Z},$ $l\geq 0.$
Then the claim follows from 
\begin{align*}
&\tilde E^{(j)}\triangleright e^p=(-1)^j\begin{bmatrix}j+p-1 \\ j
\end{bmatrix}_qe^{p+j},
\\
&\tilde F^{(i)}\triangleright e^{p+j}=\sum_{j=0}^n(-1)^jq^{-\frac{1}{2}j(j-1)+j(p+j)}\tilde F^{(n-j)}e^{p+j}\tilde F^{(j)}\subset \mathcal{F}_p(\mathcal{U}_q^{ev}).
\end{align*}
\end{remark}
\section{Examples.}\label{exam}
The Borromean tangle $B\in BT_3$ is the bottom tangle depicted in Figure \ref{fig:borromean}.
Note that $B$ is a $3$-component,  algebraically-split, $0$-framed  bottom tangle, and the  closure   of $B$ is  the Borromean rings $L_B$.
It is well known that  $L_B$ is not a ribbon link.
\begin{figure}
\centering
\includegraphics[width=3cm,clip]{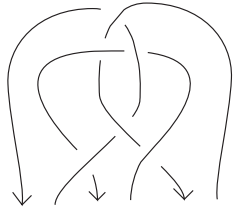}
\caption{The Borromean tangle $B\in BT_3$.}\label{fig:borromean}
\end{figure}%
 In \cite{H2}, the formulas of the universal $sl_2$ invariant of $B$ is observed;
\begin{align}
\begin{split}
&J_B=\sum_{m_1,m_2,m_3,n_1,n_2,n_3\geq 0} q^{m_3+n_3}(-1)^{n_1+n_2+n_3}q^{\sum _{i=1}^3\big(-\frac{1}{2}m_i(m_i+1)
-n_i +m_im_{i+1}-2m_in_{i-1}\big)}
\\
&\tilde F^{(n_3)}e^{m_1}\tilde F^{(m_3)}e^{n_1}K^{-2m_2}\otimes  \tilde F^{(n_1)}e^{m_2}\tilde F^{(m_1)}e^{n_2}K^{-2m_3}\otimes \tilde  F^{(n_2)}e^{m_3}\tilde F^{(m_2)}e^{n_3}K^{-2m_1}
\\
&\notin (\bar U_q^{ev})\;\hat {}^{\;\hat \otimes 3},\label{ub}
\end{split}
\end{align}
where the index $i$ should be considered modulo $3$.
The following is also observed in \cite{H2};
\begin{align}
&J_{L_B; \tilde P_i',\tilde P_j',\tilde P_k'}=\begin{cases}
(-1)^iq^{-i(3i-1)}\{ 2i+1\}_{q,i+1}/ \{ 1 \}_q \quad &\text{if} \quad i=j=k,
\\
0  &\text{otherwise.}
\end{cases}\label{qua}
\end{align}
Since $
\frac{\{ 2i+1\}_{q,i+1}}{ \{ 1 \}_q}  \notin \frac{\{ 2i+1\}_{q,i+1}}{ \{ 1 \}_q} I_iI_i$
for $i\geq 1$,  each of (\ref{ub}) and  (\ref{qua})  implies that  the Borromean rings $L_B$ is not a ribbon link.
\begin{remark}
Let $L_K$ be the $2$-component link obtained from a knot $K$ by duplicating the component.
Indeed, $L_K$ is a boundary link.
In particular,  if $K$ is a ribbon knot, then $L_K$ is a ribbon link.
We can prove
\begin{align*}
J_{L_K;\tilde P_m', \tilde P_n'}\in  \frac{\{ 2m+1\}_{q, m+1}}{\{1\} _q} I_{n}
\end{align*}
as follows.  By the formulas in Section 8 in \cite{H2}, we have 
\begin{align*}
\tilde P_m'\tilde P_n'=&\sum_{k=0}^{\min(m,n)}q^{-kl}\frac{ \{m+n \}_q! }{ \{k \}_q!\{m-k \}_q!\{n-k \}_q!}\tilde P_l'
\\
=&\sum_{k=0}^{\min(m,n)}q^{-l(k+l+1)} C_{k,m,n}(q)P_l'',
\end{align*}
where  $l=m+n-k$, $P_l''=\frac{\{1\}_q}{\{2l+1\}_{q,l+1}}q^{l(l+1)}\tilde P_l'$, and 
\begin{align*}
C_{k,m,n}(q)&=\frac{\{2m+1 \}_{q,m+1}}{\{1 \}_q} \{k\}_q!\{ n-k\} _q!
\begin{bmatrix} 2l+1 \\ 2m+1 \end{bmatrix}_q
\begin{bmatrix} 2(n-k) \\ n-k \end{bmatrix}_q
\begin{bmatrix} m+n \\ k \end{bmatrix}_q
\begin{bmatrix} m \\ k \end{bmatrix}_q 
\\
&\in  \frac{\{ 2m+1\}_{q, m+1}}{\{1\} _q} I_{n}.
\end{align*}
Theorem 6.4 in \cite{H2} implies that $J_{K;P_l''}\in \mathbb{Z}[q,q^{-1}]$ for $l\geq 0$, hence we have
\begin{align*}
J_{L_K;\tilde P_m', \tilde P_n'}=&J_{K;\tilde P_m'\cdot  \tilde P_n'}
\\
=&\sum_{k=0}^{\min(m,n)}q^{-l(k+l+1)}C_{k,m,n}(q)J_{K;P_l''}
\\
&
\in  \frac{\{ 2m+1\}_{q, m+1}}{\{1\} _q} I_{n}.
\end{align*}
\end{remark}

\begin{acknowledgments}
The author is deeply grateful to Professor Kazuo Habiro and Professor Tomotada Ohtsuki
for helpful advice and encouragement.
\end{acknowledgments}

\end{document}